\documentclass[reqno,a4paper]{elsarticle}
\journal{Arxiv}

\usepackage{amssymb,amsmath,epsfig,fancyhdr,graphics,psfrag,graphicx,color,subfigure}
\usepackage[title]{appendix}
\definecolor{lightblue}{rgb}{0.22,0.45,0.70}
\definecolor{lightgreen}{rgb}{0.1,0.6,0.10}
\usepackage[colorlinks=true,breaklinks=true,linkcolor=lightblue,citecolor=lightblue]{hyperref}

\hypersetup{
	colorlinks=true,
}

\definecolor{darkred}{rgb}{0.82,0.15,0.20}
\definecolor{darkblue}{rgb}{0.82,0.15,0.12}

\numberwithin{equation}{section}
\numberwithin{figure}{section}
\numberwithin{table}{section}
\newcommand\cero{\boldsymbol{0}}

\newcommand\bL{\mathbf{L}}
\newcommand\bI{\mathbf{I}}

\newcommand\bV{\mathbf{V}}
\newcommand\bW{\mathbf{W}}
\newcommand\bPi{\boldsymbol{\Pi}}

\newcommand\beps{\boldsymbol{\varepsilon}}

\newcommand\bb{\boldsymbol{b}}

\newcommand\bg{\boldsymbol{g}}
\newcommand\nn{\boldsymbol{n}}

\newcommand\bsigma{\boldsymbol{\sigma}}

\newcommand\bu{\boldsymbol{u}}
\newcommand\bv{\boldsymbol{v}}
\newcommand\bw{\boldsymbol{w}}
\newcommand\bx{\boldsymbol{x}}

\newcommand\RR{\mathbb{R}}
\newcommand\cT{\mathcal{T}}

\newcommand\qand{\quad\hbox{and }\quad}

\newcommand\bdiv{\mathop{\mathbf{div}}\nolimits}
\newcommand\vdiv{\mathop{\mathrm{div}}\nolimits}
\newcommand\bt{\boldsymbol{t}}

\newcommand\bnabla{\boldsymbol{\nabla}}
\newcommand\bDelta{\boldsymbol{\Delta}}

\newtheorem{remark}{Remark}[section]
\newtheorem{lemma}{Lemma}[section]
\newtheorem{theorem}{Theorem}[section]

\newtheorem{corollary}{Corollary}[section]

\newcommand\OmP{\Omega^{\mathrm{P}}}
\newcommand\OmE{\Omega^{\mathrm{E}}}
\newcommand\rmP{\mathrm{P}}
\newcommand\rmE{\mathrm{E}}

\newcommand{\dx}{\,\mbox{d}x}

\newcommand{\ds}{\,\mbox{d}s}

\newenvironment{proof}{\noindent{\it Proof.}}{\hfill$\square$}

\allowdisplaybreaks

\begin{document}
	
	\hypersetup{
		linkcolor=lightgreen,
		urlcolor=lightgreen,
		citecolor=lightgreen
	}
	\begin{frontmatter}
		
		\title{\textbf{Numerical solution of the Biot/elasticity \\interface problem using virtual element methods}}
	
		\author[iist]{Sarvesh Kumar}
		\ead{sarvesh@iist.ac.in}
		
		\author[ubb,ci2ma]{David Mora}
		\ead{dmora@ubiobio.cl}
		
		\author[monash,unach]{Ricardo Ruiz-Baier}
		\ead{ricardo.ruizbaier@monash.edu}
		
		\author[iitb]{Nitesh Verma}
		\ead{p15379@iitb.ac.in}
		
		\address[iitb]{Department of Mathematics, Indian Institute of Technology Bombay, Powai, Mumbai 400 076, 
Maharashtra, India.}

		\address[iist]{Department of Mathematics, Indian Institute of Space	Science and Technology, Trivandrum 695 547, India.}
			
		\address[ubb]{GIMNAP, Departamento de Matem\'atica, Universidad del B\'io-B\'io, Casilla 5-C, Concepci\'on, Chile.}
		\address[ci2ma]{CI$^{\,2}\!$MA, Universidad de Concepci\' on, Chile.}	
		
		\address[monash]{School of Mathematics, Monash University, 9 Rainforest walk, Melbourne, 3800 VIC, Australia.}
		\address[unach]{Universidad Adventista de Chile, Casilla 7-D, Chill\'an, Chile.}
		
		
		\begin{abstract}
We propose, analyse and implement a virtual element discretisation for an interfacial poroelasticity--elasticity consolidation problem. The formulation of the time--dependent poroelasticity equations uses displacement, fluid pressure and total pressure, and the elasticity equations are written in displacement-pressure formulation. The construction of the virtual element scheme does not require Lagrange multipliers to impose the transmission conditions (continuity of displacement and total traction, and no--flux for the fluid) on the interface.  We show the stability and convergence of the virtual element method for different polynomial degrees, and the error bounds are robust with respect to delicate model parameters (such as Lam\'e constants, permeability, and storativity coefficient). Finally we provide numerical examples that illustrate the properties of the scheme.

		\end{abstract}
		
		\begin{keyword} Biot equations \sep virtual element methods \sep time-dependent problems \sep a priori error analysis.
			\MSC 65M60 \sep 74F10 \sep 35K57 \sep 74L15.
		\end{keyword}
		
	\end{frontmatter}

\section{Introduction and problem statement}

Diverse applications in the fields of environmental engineering, geoscience and biomedicine involve the modelling of transmission/contact poroelastic-elastic processes. Examples of such applications include land subsidence, solid waste management, withdrawal of ground water, harnessing of geothermal energy, describing blood perfusion of deformable living tissues (such as cartilage), and many other scenarios of increasing complexity \cite{girault15,hossein21}. Obtaining approximate solutions for the coupling of multiphysics problems through an interface typically requires to combine numerical methods of distinct nature, coming from,  e.g., different mesh sizes and types, and exhibiting hanging nodes across the interface. In particular, when one of the multiphysics sub-problems is constituted by the elasticity equations, a well-known issue is the presence of volumetric locking observed in the nearly incompressible regime; and several different remedies are available, including mixed formulations of diverse types, enrichment, and reconstructions (see, e.g., \cite{herrmann65} and the monograph \cite{boffi13}). In the applications mentioned above, the elasticity equations interact with the quasi-static Biot's equations for poroelasticity. There, two additional key issues are commonly encountered with classical discretisations:  one is Poisson locking (when the Poisson ratio approaches 1/2) and the other is the presence of spurious pressure oscillations occurring when the storativity coefficient approaches zero \cite{mardal21}. Methods that enjoy the property of being locking free are abundant in the literature, including mixed finite elements (FE) \cite{boon21,hong18,lee17,lee19,oyarzua16,rodrigo16,yi17}, discontinuous Galerkin (DG) \cite{chen13,kanschat18,riviere17,zhao21}, hybrid high-order (HHO) \cite{boffi16,botti21,fu19}, finite volume-element (FVE) \cite{kumar19,rl16}, and some types of virtual element methods (VEM). From the latter class of schemes, we mention \cite{coulet20} where the authors address robust discretisations for elasticity, the families of locking-free VE methods for three-field poroelasticity introduced in  \cite{burger_acom21,tang21}, and the extended theory amenable for heterogeneous domains recently proposed in \cite{skreek21}.

The interfacial Biot/elasticity problem has been studied in great detail (analysis of unique solvability, design and analysis of finite element discretisations, and derivation of error estimates) in the following references \cite{adgmr20,anaya22,badia23,gira11,gira20,girault15}. Its formulation requires a careful setup and analysis of transmission conditions. In these contributions the meshes from the two subdomains are assumed to match on the interface. Here we conduct a similar analysis, but focusing in the case of VEM, which in particular removes the restriction on mesh conformity at the interface, and it also permits us to have small edges. Formulations for interface problems using VEM have been analysed in the case of elliptic transmission equations \cite{cww_jcp16}, for example. In such cases the analysis needs additional mesh conditions on the interface, which might be too restrictive in view of the applicative Biot/elasticity systems where fine physical details need to be captured near the interface.

Typical VEM formulations feature a consistency term that involves suitably scaled projections. There exist recent VEM tools advanced in \cite{daveiga17,brenner18} that use a different stabilisation (based on edge/face tangential jumps) which relaxes the mesh assumptions, and a similar idea is also put forward in the context of HHO methods \cite{droniou21} (addressing several types of possible stabilisers). An important observation, however, is that, at least in the cases studied herein, numerical results indicate that those types of alternative stabilisations do not present significant differences (in terms of experimental convergence rates) with respect to the most widely adopted VEM stabilisation (the \emph{dofi-dofi} stabiliser \cite{daveiga-b13}), even in the case of meshes with small edges and for polynomial degrees larger or equal than 2. Thus the advantage in this scenario seems to be confined to the convenience in the theoretical analysis. We therefore restrict the analysis under standard mesh assumptions for VEM as in, e.g., \cite{daveiga-NS18}, treating the case of partitions with close (but not collapsing) vertices only from the numerical viewpoint. We also specify that the VEM advanced herein is based on a new formulation for the elasticity--poroelasticity coupling using three fields: a global displacement, the fluid pressure, and a global pressure. The analysis can be framed following an extension of the VEM for three-field Biot's poroelasticity  recently proposed in \cite{burger_acom21}.

The outline of this work is as follows. The well-posed weak formulation of the interface poroelasticity-elasticity problem is defined in Section \ref{sec:model}. In Section \ref{sec:VEapprox}, we define a new VE formulation for the interface problem, and we then derive the well-posedness of both semi and fully discrete schemes. Section~\ref{sec:estimates} contains the a priori error analysis for the time and space discretisation. We close in Section \ref{sec:numer} by presenting simple numerical experiments that serve to verify the theoretically predicted optimal order of convergence. For sake of readability, details on the proofs of the four main theorems of the paper are collected in the Appendix.

\section{Equations of time-dependent interface poroelasticity/elasticity with global  pressure}\label{sec:model}
Considering the problem formulation presented in the references \cite{gira11,gira20}, we examine a bounded Lipschitz domain $\Omega\subset\mathbb{R}^d$, where $d \in \{2,3\}$. This domain is divided into two non-overlapping and connected subdomains, $\OmE$ and $\OmP$, representing regions of non-pay rock (treated as an elastic body) and a poroelastic reservoir, respectively. The interface between these subdomains is denoted as $\Sigma=\partial\OmP\cap \partial\OmE$, with the outward normal vector $\nn$ pointing from $\OmP$ to $\OmE$. The boundary of the domain $\Omega$ is further separated into the boundaries of the individual subdomains: $\partial \Omega:=\Gamma^\mathrm{P} \cup \Gamma^\mathrm{E}$. These boundaries are then divided into disjoint Dirichlet and Neumann conditions: $\Gamma^\mathrm{P}:= \Gamma^\mathrm{P}_D \cup \Gamma^\mathrm{P}_N$ and $\Gamma^\mathrm{E}:= \Gamma^\mathrm{E}_D \cup \Gamma^\mathrm{E}_N$, respectively.

In the entire domain, we formulate the momentum balance for the poroelastic region, the conservation of mass for the fluid, and the linear momentum balance for the elastic region. Following the approach in references \cite{adgmr20,anaya22}, we introduce additional variables, namely the total pressure in the poroelastic subdomain and the Herrmann pressure in the elastic subdomain. In addition to the conventional variables of elastic displacement, poroelastic displacement, and fluid pressure, we seek to determine the vector of solid displacements $\bu^\mathrm{E}:\OmE\to \RR^d$ for the non-pay zone, the elastic pressure $\psi^\mathrm{E}:\OmE\to\RR$, the displacement $\bu^\mathrm{P}(t):\OmP\to \RR^d$, the pore fluid pressure $p^\mathrm{P}(t):\OmP\to\RR$, and the total pressure $\psi^\mathrm{P}(t):\OmP\to\RR$ in the reservoir. These variables should satisfy the following equations for each time $t\in (0,t_{\mathrm{final}}]$, given the body loads $\bb^\mathrm{P}(t):\OmP\to \RR^d$, $\bb^\mathrm{E}(t):\OmE\to \RR^d$, and the volumetric source or sink $\ell^\mathrm{P}(t):\OmP\to \RR$:

\begin{figure}[!t]
\begin{center}
\includegraphics[width=0.44\textwidth]{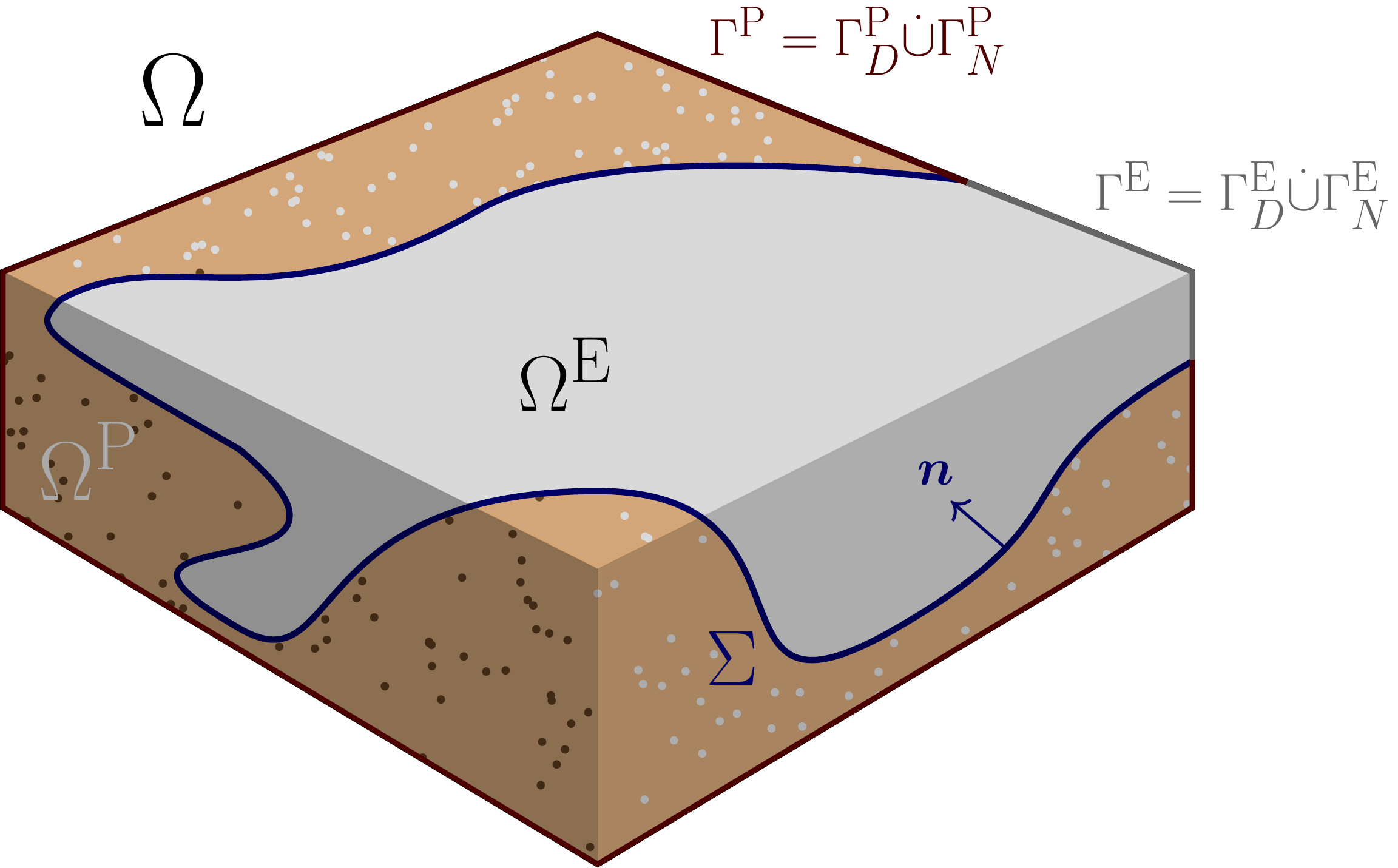}
\end{center}

\vspace{-3mm}
\caption{Sketch of the multidomain configuration including sub-domains,
sub-boundaries, and interface. The poroelastic boundary $\Gamma^{\mathrm{P}}$
is further decomposed into $\Gamma^{\mathrm{P}}_D$ and $\Gamma^{\mathrm{P}}_N$;
likewise the elastic boundary $\Gamma^{\mathrm{E}}$ is split
between $\Gamma^{\mathrm{E}}_D$ and $\Gamma^{\mathrm{E}}_N$.}\label{fig:sketch}
\end{figure}

\begin{subequations}\label{eq:coupled}
\begin{align}
-\bdiv( 2{\mu^\mathrm{P}} \beps(\bu^\mathrm{P})- \psi^\mathrm{P} \bI)&
= \bb^\mathrm{P} & \text{in $\OmP\times(0,t_{\mathrm{final}}]$},\label{eq:poro-momentum}\\
\biggl(c_0
+\frac{\alpha^2}{{\lambda^\mathrm{P}}}\biggr) \partial_t p^\mathrm{P} -\frac{\alpha}{{\lambda^\mathrm{P}}}  \partial_t \psi^\mathrm{P}
- \frac{1}{\eta} \vdiv(\kappa  \nabla p^\mathrm{P} ) &= \ell^\mathrm{P}
& \text{in $\OmP\times(0,t_{\mathrm{final}}]$},\label{eq:Biot} \\
\psi^\mathrm{P} - \alpha p^\mathrm{P} + {\lambda^\mathrm{P}} \vdiv\bu^\mathrm{P} &=
0 & \text{in $\OmP\times(0,t_{\mathrm{final}}]$},\label{eq:poro-constitutive}\\
-\bdiv( 2{\mu^\mathrm{E}} \beps(\bu^\mathrm{E})- \psi^\mathrm{E} \bI)&
= \bb^\mathrm{E} & \text{in $\OmE\times(0,t_{\mathrm{final}}]$},\label{eq:Elast}\\
\psi^\mathrm{E} + {\lambda^\mathrm{E}}\vdiv\bu^\mathrm{E} &=
0 & \text{in $\OmE\times(0,t_{\mathrm{final}}]$}. \label{eq:elast-constitutive}
\end{align}\end{subequations}

In this formulation, the hydraulic conductivity of the porous medium is denoted as $\kappa(\bx)$, while $\eta$ represents the constant viscosity of the interstitial fluid. The storativity coefficient $c_0$ and the Biot--Willis consolidation parameter $\alpha$ are involved as well. Additionally, $\mu^\mathrm{E},\lambda^\mathrm{E}$ and $\mu^\mathrm{P},\lambda^\mathrm{P}$ are the Lamé parameters associated with the constitutive law of the solid in the elastic and poroelastic subdomains, respectively.

The poroelastic stress $\widetilde{\bsigma}$ is defined as the difference between the effective mechanical stress $\bsigma$ and the product of the poroelastic pressure $p^{\rmP}$ and the identity tensor $\bI$, scaled by the parameter $\alpha$. The effective mechanical stress $\bsigma$ consists of the term ${\lambda^\mathrm{P}} (\vdiv \bu^{\rmP})\bI + 2{\mu^\mathrm{P}} \beps(\bu^{\rmP})$, which represents the mechanical stress due to deformation in the poroelastic subdomain, and the non-viscous fluid stress, which is proportional to the fluid pressure and scaled by $\alpha$.

The system of equations is supplemented by mixed-type boundary conditions
\begin{subequations}
\begin{align}
\label{bc:Gamma}
\bu^\mathrm{E} = \cero & &\text{on $\Gamma^\mathrm{E}_D\times(0,t_{\text{final}}]$},\\
\label{bc:Sigma}
[{\lambda^\mathrm{E}}(\vdiv \bu^\mathrm{E})\bI + 2{\mu^\mathrm{E}} \beps(\bu^\mathrm{E})]\nn^{\Gamma} = \cero
& &\text{on $\Gamma^\mathrm{E}_N\times(0,t_{\text{final}}]$},\\
\label{bc:GammaP}
\bu^\mathrm{P} = \cero \quad \text{and} \quad \frac{\kappa}{\eta} \nabla p^\mathrm{P} \cdot\nn^{\Gamma} = 0  & &\text{on $\Gamma^\mathrm{P}_D\times(0,t_{\text{final}}]$},\\
\label{bc:SigmaP}
\widetilde{\bsigma}\nn^{\Gamma} = \cero \quad  \text{and}\quad p^\mathrm{P}=0
& &\text{on $\Gamma^\mathrm{P}_N\times(0,t_{\text{final}}]$},
\end{align}\end{subequations}
transmission conditions that ensure continuity of the medium, the balance of total tractions
 and no flux of fluid at the interface:
 \begin{equation}\label{eq:transmission2}
\bu^\mathrm{P} = \bu^\mathrm{E}, \quad
[2{\mu^\mathrm{P}} \beps(\bu^\mathrm{P})- \psi^\mathrm{P} \bI]\nn
=[2{\mu^\mathrm{E}} \beps(\bu^\mathrm{E})- \psi^\mathrm{E} \bI]\nn, \quad
\frac{\kappa}{\eta}\nabla p^\mathrm{P}\cdot\nn = 0
\quad \text{on $\Sigma\times(0,t_{\text{final}}]$}.
\end{equation}
 The initial conditions for the system are given as $p^{\rmP}(0) =0$ and $\bu^{\rmP}(0) = \cero$ in the domain $\Omega^\mathrm{P}$.

It is important to note that the assumption of homogeneity in the boundary and initial conditions is made solely for the purpose of simplifying the subsequent analysis.
Let us now consider  the following functional  spaces
$$\bV^\mathrm{E}:=[H^1_{\Gamma_D^\rmE}(\Omega^\mathrm{E})]^d,
\quad \bV^\mathrm{P}:=[H^1_{\Gamma_D^\rmP}(\Omega^\mathrm{P})]^d,
\quad Q^\mathrm{P}:= H^1_{\Gamma_D^\rmP}(\OmP),\quad
Z^\mathrm{E}:=L^2(\OmE),\quad Z^\mathrm{P}:=L^2(\OmP).$$

Multiplying \eqref{eq:Biot} by adequate test functions, integrating
by parts (in space) whenever appropriate, and using the boundary
conditions \eqref{bc:Gamma}-\eqref{bc:Sigma}, leads to the following
 weak formulation: For a given $t>0$, find $\bu^\mathrm{P}(t) \in \bV^\mathrm{P},
\bu^\mathrm{E}(t) \in \bV^\mathrm{E}, p^\mathrm{P}(t) \in Q^\mathrm{P},
\psi^\mathrm{E}(t) \in Z^\mathrm{E}, \psi^\mathrm{P}(t) \in Z^\mathrm{P}$ such that
\begin{align*}
2{\mu^\mathrm{P}} \int_{\OmP} \beps(\bu^\mathrm{P}):\beps(\bv^\mathrm{P})
-\int_{\OmP}\psi^\mathrm{P}\vdiv \bv^\mathrm{P}
-\int_{\partial\OmP}[2{\mu^\mathrm{P}} \beps(\bu^\mathrm{P})
- \psi^\mathrm{P}\bI]\nn\cdot\bv^\mathrm{P}  &= \int_{\OmP}\bb^\mathrm{P}\cdot\bv^\mathrm{P},\\
\int_{\OmP}\psi^\mathrm{P}\phi^\mathrm{P}-\alpha\int_{\OmP}p^\mathrm{P}\phi^\mathrm{P}
+{\lambda^\mathrm{P}}\int_{\OmP}\phi^\mathrm{P}\vdiv \bu^\mathrm{P} &=0,\\
\biggl( c_0 +\frac{\alpha^2}{{\lambda^\mathrm{P}}}\biggr)\int_{\OmP}  \partial_t p^\mathrm{P} q^\mathrm{P}
-\frac{\alpha}{{\lambda^\mathrm{P}}} \int_{\OmP} \partial_t \psi^\mathrm{P}q^\mathrm{P}
+\frac{1}{\eta}\int_{\OmP}\kappa \nabla p^\mathrm{P}\cdot\nabla q^\mathrm{P}
-\frac{1}{\eta}\int_{\partial\OmP}(\kappa \nabla p^\mathrm{P}\cdot\nn)q^\mathrm{P}
&=\int_{\OmP}\ell^{\mathrm{P}}q^\mathrm{P},\\
2{\mu^\mathrm{E}}\int_{\OmE} \beps(\bu^\mathrm{E}):\beps(\bv^\mathrm{E})
-\int_{\OmE}\psi^\mathrm{E}\vdiv \bv^\mathrm{E}
-\int_{\partial\OmE}[2{\mu^\mathrm{E}} \beps(\bu^\mathrm{E})
- \psi^\mathrm{E}\bI]\nn^{\partial\OmE}\cdot\bv^\mathrm{E}&=\int_{\OmE}\bb^\mathrm{E}\cdot\bv^\mathrm{E},\\
\int_{\OmE}\psi^\mathrm{E}\phi^\mathrm{E} +{\lambda^\mathrm{E}}\int_{\OmE}\phi^\mathrm{E}\vdiv \bu^\mathrm{E} &=0.
\end{align*}

Thanks to the weak formulation above and to \eqref{eq:transmission2}, we can use a {global displacement} $\bu\in\bV:=[H^1_{\Gamma_D}(\Omega)]^d$, which satisfies $\bu|_{\OmP}=\bu^\mathrm{P}$ and $\bu|_{\OmE}=\bu^\mathrm{E}$. Similarly, we can define a {global total pressure} $\psi\in Z:=L^2(\Omega)$, which combines the total pressure in the poroelastic medium and the elastic hydrostatic pressure in the elastic subdomain, that is $\psi|_{\OmP}=\psi^\mathrm{P}$ and $\psi|_{\OmE}=\psi^\mathrm{E}$. 
Moreover, we extend this approach to define the body load $\bb\in [L^2(\Omega)]^d$, which comprises $\bb|_{\OmP}=\bb^\mathrm{P}$ and $\bb|_{\OmE}=\bb^\mathrm{E}$. Additionally, we introduce the global Lam\'e parameters $\mu$ and $\lambda$, where $\mu$ and $\lambda$ satisfy $\mu|_{\OmP}=\mu^\mathrm{P},\,\lambda|_{\OmP}=\lambda^\mathrm{P}$ and $\mu|_{\OmE}=\mu^\mathrm{E},\,\lambda|_{\OmE}=\lambda^\mathrm{E}$.

By combining the aforementioned steps with the second and third transmission conditions in \eqref{eq:transmission2}, we arrive at the following weak formulation: Find $\bu(t) \in \bV$, $p^\mathrm{P}(t) \in Q^\mathrm{P}$, and $\psi(t)\in Z$ such that
\begin{subequations}
\begin{alignat}{5}
&&   a_1(\bu,\bv)   &&                 &\;+&\; b_1(\bv,\psi)     &=&\;F(\bv)&\quad\forall \bv\in\bV, \label{weak-u}\\
{\tilde{a}_2(\partial_t p^\mathrm{P},q^\mathrm{P})} &\; +&               &&      a_2(p^\mathrm{P},q^\mathrm{P})   &\;-&\;    b_2(q^\mathrm{P}, \, \partial_t \psi) &=&\;G(q^\mathrm{P}) &\quad\forall q^\mathrm{P}\in Q^\mathrm{P}, \label{weak-p}\\
&&b_1(\bu,\phi)  &\;+\;& b_2(p^\mathrm{P},\phi)&\;-&\; a_3(\psi,\phi) &=&0 &\quad\forall\phi\in Z, \label{weak-psi}
\end{alignat}\end{subequations}
where the bilinear forms
$a_1:\bV\times\bV \to \RR$,
$a_2:Q^\mathrm{P}\times Q^\mathrm{P} \to \RR$, $a_3: Z\times Z\to \RR$,
$b_1:\bV \times Z \to \RR$, $b_2:Q^\mathrm{P} \times Z\to \RR$,
and linear functionals $F:\bV \to\RR$, $G:Q^\mathrm{P}\to\RR$, adopt the following form
\begin{gather*}
a_1(\bu,\bv):=  { \int_{\Omega} 2 \mu~ \beps(\bu):\beps(\bv)}, \qquad 
b_1(\bv,\phi):= -\int_{\Omega}\phi\vdiv \bv, \qquad 
F(\bv) := \int_{\Omega} \bb \cdot\bv,\\
\tilde{a}_2(p^\mathrm{P},q^\mathrm{P})  := \biggl( c_0 +\frac{\alpha^2}{{\lambda^\mathrm{P}}}\biggr)
\int_{\Omega^\mathrm{P}}  p^\mathrm{P} q^\mathrm{P} , \qquad 
a_2(p^\mathrm{P},q^\mathrm{P})  :=\frac{1}{\eta}\int_{\Omega^\mathrm{P}}\kappa \nabla p^\mathrm{P}\cdot\nabla q^\mathrm{P},
\\
b_2(p^\mathrm{P},\phi):=  \frac{\alpha}{{\lambda^\mathrm{P}}} \int_{\Omega^\mathrm{P}} p^\mathrm{P}\phi^\mathrm{P}, \qquad
a_3(\psi,\phi):=
 {\int_{\Omega} \frac{1}{\lambda}~ \psi \, \phi,}
\qquad
G(q^\mathrm{P}) := \int_{\Omega^\mathrm{P}} \ell^\mathrm{P} \; q^\mathrm{P}. \nonumber
\end{gather*}
\section{Virtual element approximation} \label{sec:VEapprox}

\subsection{Discrete spaces and degrees of freedom} \label{subsec:VEspaces}
In this section we construct a VEM associated with \eqref{weak-u}-\eqref{weak-psi}.
The main ingredients are detailed for the 2D case, and later in Section~\ref{subsec:3D}
we outline the required modifications that allow the extension to 3D.

We denote by  $\{{\mathcal T}_h^\mathrm{P}\}_h$ and $\{{\mathcal T}_h^\mathrm{E}\}_h$
sequences of polygonal decompositions of the poroelastic and elastic subdomains $\Omega^\mathrm{P}$ and $\Omega^\mathrm{E}$, respectively.
The union of the two non-conforming partitions might have hanging nodes
on the elements at interface. Such polygonal elements are replaced with new elements
same as the original besides the hanging nodes simply regarded as  additional nodes.
The  notation for the union of all the elements with new interfacial
elements is $\{{\mathcal T}_h\}_h$, having mesh-size $h:=\max_{K\in{\mathcal T}_h}h_K$.
By $N^v_K$ we will denote the number of vertices in the polygon $K$, $N^e_K$
will stand for the number of edges on $\partial K$, and $e$ a generic
edge of $\mathcal{T}_h$. For all $e\in \partial K$, we denote by $\boldsymbol{n}_K^e$
the unit normal pointing outwards $K$, $\bt^e_K$ the unit tangent vector along $e$
on $K$, and $V_i$ represents the $i^{th}$ vertex of the polygon $K$.
As in \cite{daveiga-b13} we  suppose regularity of the polygonal meshes in the following sense:
there exists $C_{{\mathcal T}}>0$ such that, for every $h$ and every $K\in {\mathcal T}_h$,
the following holds
	\begin{itemize}
		\item[($A$)] $K\in{\mathcal T}_h$ is star-shaped with respect to every point within a ball of radius $C_{{\mathcal T}}h_K$;
		\item[($\tilde{A}$)] the ratio between the shortest edge and $h_K$ is larger than $C_{{\mathcal T}}$.
	\end{itemize}
	
\begin{remark}\label{hfgrtrdc}
We note that both assumptions above are required
by the subsequent error analysis. It is possible to ignore
assumption $(\tilde{A})$, however resulting in much more
involved proofs that follow the technical tools advanced
in \cite{daveiga17,brenner18} for elliptic problems.
Also, the numerical experiments from Section \ref{sec:numer}
illustrate that even without $(\tilde{A})$,
the accuracy and performance of the method is not
affected by meshes featuring relatively small edges.
\end{remark}

Denoting by $\mathbb{P}_k(K)$ the space of polynomials of degree up to $k$, defined locally on $K\in\cT_h$,
we proceed to characterise the scalar energy projection operator $\Pi_{K}^{\nabla,k}: H^1(K) \rightarrow \mathbb{P}_k(K)$ by the relations
\begin{equation*}
(\nabla (\Pi_{K}^{\nabla,k} q - q), \nabla r_k)_{0,K} = 0, \qquad P^0_K(\Pi_{K}^{\nabla,k} q - q)=0,
\end{equation*}
valid for all $q \in H^1(K)$ and $r_k \in \mathbb{P}_k(K)$, and where
$ P^0_K(q):= \int_{ K} q\;\dx$.
 Next, the vectorial energy operator $\bPi^{\bnabla,k}_{K}: [H^1(K)]^2 \rightarrow [\mathbb{P}_k(K)]^2$ is described as follows
\[(\bnabla (\bPi^{\bnabla,k}_{K} \bv - \bv), \bnabla \mathbf{p}_k)_{0,K} =0,  \qquad
\bPi^{0,0}_K(\bPi^{\bnabla,k}_{K} \bv - \bv) = 0,
\]
for all $\mathbf{p}_k \in [\mathbb{P}_k(K)]^2, \bv \in [H^1(K)]^2$, and $\bPi^{0,0}_K$ is the $L^2$-projection onto constants.

Denoting by $\mathcal{M}_k(K)$ the space of monomials of degree up to $k$  defined locally on  $K$, we can define the local VE spaces for global displacement, fluid pressure, and global pressure for $k \ge 2$ as follows
\begin{align} \label{VE-spaces}
\begin{split}
\bV_h^k(K) & :=   { \Big\{ \bv \in \mathbb{B}^k(\partial K):
\begin{cases}
- \bDelta \bv + \nabla s \in \mathcal{G}_{k-2}^{\perp}(K) ~ \text{ for some } s \in L^2(K) \\
\vdiv \bv \in \mathbb{P}_{k-1}(K),
\end{cases}\Big\}, } \\
Q_h^{k,\rmP}(K) & := \bigl\{ q_h^{\rmP} \in H^1(K) \cap C^0(\partial K): \Delta q_h^{\rmP}|_{K} \in \mathbb{P}_k(K), \, q_h^{\rmP}|_e \in \mathbb{P}_k(e), \forall e \in \partial K, \\ & \quad \quad  \qquad\quad \qquad \quad  \qquad \qquad \; (\Pi_{K}^{\nabla,k} q_h^{\rmP} - q_h^{\rmP}, m_{\alpha})_{0,K} = 0\ \forall m_{\alpha} \in \mathcal{M}_k \backslash\mathcal{M}_{k-2}(K) \bigr\}, \\
Z_h^k(K) & := \mathbb{P}_{k-1}(K),
\end{split} 	\end{align}
where the element boundary space and an auxiliary space are given as
\begin{align*}
\mathbb{B}^k (\partial K) := \Bigl\{ \bv_h \in [H^1(K) \cap C^0(\partial K)]^2: \bv_h|_{e} \in [\mathbb{P}_k(e)]^2 \quad \forall e \in \partial K \Bigr\},\qquad \mathcal{G}_{k}(K):= \bnabla \mathbb{P}_{k+1}(K).
\end{align*}
Note that $K \in \cT_h$ for the local spaces $\bV_h^k(K)$ and $Z_h^k(K)$, and $K \in \cT_h^\rmP$ for space $Q_h^{k,\rmP}(K)$.

The dimension of $\bV_h^k(K)$ is $2 k N^v_K + \frac{(k-1)(k-2)}{2} + \frac{k(k+1)}{2} - 1$, the dimension of $Q_h^{k,\rmP}(K)$ is $k N^v_{K} + \frac{k(k-1)}{2}$, and that of $Z_h^k(K)$ is $\frac{k(k+1)}{2}$. The fluid pressure approximation will follow the VE spaces of degree $k\ge 2$ from \cite{ahmad13}. This facilitates the computation of the $L^2$-projection onto the space of polynomials of degree up to $k$ (which are required in order to define the zero-order discrete bilinear form on $Q_h^{k,\rmP}(K)$).

Next we specify the degrees of freedom associated with \eqref{VE-spaces}. That is, discrete functionals of the type (taking as an example the space for global pressure)
$$(D_i): Z_{h|K} \to \mathbb{R}; \qquad  Z_{h|K} \ni \phi \mapsto D_i(\phi),$$
and we start with the degrees of freedom for the local displacement space $\bV_h^k(K)$, consisting of
($D_v1$) the values of a discrete displacement $\bv_h$ at vertices of the element;
($D_v2$) the values of $\bv_h$ at $(k-1)$ distinct internal Gauss--Lobatto points;
($D_v3$) the moments
	$\frac{1}{|K|} \int_{K} \bv_h \cdot \bg_{k-2}^{\perp}, \quad \forall \bg_{k-2}^{\perp} \in \mathcal{G}_{k-2}^{\perp}(K)$;
and ($D_v4$) the moments
	$\frac{1}{|K|} \int_{K} \vdiv \bv_h~ q_{k-1}, \quad \forall  q_{k-1} \in \mathbb{P}_{k-1}(K) \backslash \mathbb{R}.$

Then we precise the degrees of freedom for the local fluid pressure space $Q_h^{k,\rmP}(K)$:
($D_q1$) the values of $q_h^{\rmP}$ at vertices of the polygonal element;
($D_q2$) the values of $q_h^{\rmP}$ at $(k-1)$ distinct internal Gauss--Lobatto points;
and ($D_q3$) the moments of $q_h^{\rmP}$
$\frac{1}{|K|} \int_{\Omega} q_h^{\rmP}~ m_{\alpha}, \quad \forall m_{\alpha} \in \mathcal{M}_{k-2}(K).$
And similarly, the degrees of freedom for the local global-pressure space $Z_h^k(K)$:
 ($D_z$) the moments:
	$\frac{1}{|K|} \int_K \phi_h~ m_{\alpha}, \quad \forall m_{\alpha} \in \mathcal{M}_{k-1}(K).$

It has been proven elsewhere (e.g.  \cite{ahmad13,daveiga-div17,daveiga-NS18}) that these degrees of
freedom are unisolvent in their respective spaces. We also define global counterparts of the local VE spaces, as follows
\begin{gather*}
\bV_h^k : = \lbrace \bv_h \in \bV : \bv_h |_K \in \bV_h^k(K) \; \forall K \in \mathcal{T}_h \rbrace, \\
Q_h^{k,\rmP} := \lbrace q_h^{\rmP} \in Q^{\rmP} : q_h^{\rmP} |_K \in Q_h^{k,\rmP}(K) \; \forall K \in \mathcal{T}_h^{P} \rbrace,\qquad Z_h^k := \lbrace \phi_h \in Z : \phi_h |_K \in Z_h^k(K) \; \forall K \in \mathcal{T}_h \rbrace.
\end{gather*}
In addition,
$N^\bV$ denotes the number of degrees of freedom for $\bV_h^k$, $N^Q$ the number of degrees of freedom for $Q_h^{k,\rmP}$,
$N^Z$ the number of degrees of freedom for $Z_h^k$,
and $\text{dof}_r(s)$ stands for the $r$-th degree of a given function $s$.

\subsection{Projection operators}
We will use for a generic bilinear form $\aleph(\cdot,\cdot)$, the notation
$\aleph^K(\cdot,\cdot) = \aleph(\cdot,\cdot)|_K.$
 Then we can define the  energy projection $\bPi^{\beps,k}_K : \bV_h^k(K) \rightarrow [\mathbb{P}_k(K)]^2$ such that,
\begin{align*}
& (\bPi^{\beps,k}_K \bv - \bv , \mathbf{r}_k)_{\beps,K} = 0 , \quad m^K(\bPi^{\beps,k}_K \bv - \bv,\mathbf{r}_k) = 0, \quad \forall \bv \in \bV_h^k(K), \mathbf{r}_k \in [\mathbb{P}_k(K)]^2,
\end{align*}	
where
\[
 (\bu, \bv)_{\beps,K} :=\int_K \beps(\bu): \beps (\bv),  \quad
 m^K (\bv,\mathbf{r}_k):= \frac{1}{N^v_K} \sum_{i=1}^{N^v_K} \bv(V_i) \cdot \mathbf{r}_k(V_i), \quad \forall \bu, \bv \in \bV_h^k(K), \ \mathbf{r}_k \in \text{ker}(a_1^K(\cdot,\cdot)).
\]
Then, using the degree of freedom $(D_v1)$, we can readily compute the bilinear form $m^K(\bv, \mathbf{r}_k)$, for all $\mathbf{r}_k \in$ ker($a_1^K(\cdot, \cdot)$), $\bv \in \bV_h^k(K)$.
Next, for all $\bv \in \bV_h^k(K)$, let us consider the localised form
\begin{align*}
a_1^K(\bv, \mathbf{r}_k) = \int_K \beps (\bv): \beps (\mathbf{r}_k) = -\int_K \bv \cdot \bdiv ( \beps(\mathbf{r}_k)) + \int_{\partial K} \bv \cdot(\beps(\mathbf{r}_k) \nn^e_K) .
\end{align*}
One readily sees that $\bdiv(\beps(\mathbf{r}_k))  \in [\mathbb{P}_{k-2}(K)]^2$ and $\beps(\mathbf{r}_{k}) \in [\mathbb{P}_{k-1}(K)]^{2 \times 2}$ for all $\mathbf{r}_k \in [\mathbb{P}_k(K)]^2$, and the second term on the right-hand side is computed by the Gauss--Lobatto rule with the use of degrees of freedom $(D_v1)$-$(D_v2)$. Also, the vectorial polynomial space can be split in terms of the orthogonal space $\mathcal{G}_{k}(K)$ and its complement, that is $$[\mathbb{P}_k(K)]^2 = \mathcal{G}_{k}(K) + \mathcal{G}_{k}^{\perp}(K) = \bnabla \mathbb{P}_{k+1}(K) + \mathcal{G}_{k}^{\perp}(K).$$
Starting from $\bdiv \beps(\mathbf{r}_k) = \nabla \mathbf{p}_{k-1} + \bg_{k-2}^{\perp}$ for some $\mathbf{p}_{k-1} \in \mathbb{P}_{k-1}(K)$, $\bg_{k-2}^{\perp} \in \mathcal{G}_{k-2}^{\perp}(K)$,  an integration by parts gives
\begin{align*}
\int_K \bv \cdot \bdiv ( \beps(\mathbf{r}_k))
= - \int_K \vdiv \bv \, \mathbf{p}_{k-1} + \int_{\partial K} (\bv \cdot \nn_K) \mathbf{p}_{k-1}  + \int_K \bv \cdot \bg_{k-2}^{\perp}.
\end{align*}
The terms on the right-hand side are computed through the explicit evaluation of $\vdiv \bv$ for $\bv \in \bV_h^k(K)$ by making use of the degrees of freedom $(D_v1)$-$(D_v2)$, $(D_v4)$ for the first term; appealing only to $(D_v1)$-$(D_v2)$ to determine $\bv$ on the boundary of each $K$, while the third term is calculated trivially from $(D_v3)$. The precise action of the projection operators $\bPi^{\bnabla,k}_{K}, \Pi^{\nabla,k}_{K}$ (used in the definition of the VE spaces), has been shown in detail in \cite{ahmad13,daveiga-NS18}.

We now define the $L^2$-projection on the scalar space as $\Pi^{0,k}_K: L^2(K) \rightarrow \mathbb{P}_k(K)$ such that
\begin{align*}
(\Pi_{K}^{0,k} q - q, r_k)_{0,K} = 0, \quad q \in L^2(K), r_k \in \mathbb{P}_k(K),
\end{align*}
and we can clearly see that it is computable on the space $Q_h^{k,\rmP}$.

Finally, we consider the $L^2$-projection onto piecewise vector-valued {polynomials},
$\,\smash{\bPi^{0,k}_K}: [L^2(K)]^2  \rightarrow  [\mathbb{P}_{k}(K)]^2$, which is   fully computable on $\bV_h^k(K)$.

\subsection{Discrete bilinear forms and formulations}
For all $\bu_h, \bv_h \in \bV_h^k(K)$ and $p_h^{\rmP}, q_h^{\rmP} \in Q_h^{k,\rmP}(K)$ we now
define the  local discrete bilinear forms
\begin{align*}
a_1^h(\bu_h, \bv_h) |_K &:= a_1^K(\bPi^{\beps,k}_K \bu_h, \bPi^{\beps,k}_K \bv_h) + S_1^K \bigl((\bI-\bPi^{\beps,k}_K)\bu_h, (\bI-\bPi^{\beps,k}_K)\bv_h \bigr),\\
a_2^h(p_h^{\rmP},q_h^{\rmP})|_K &:=
 { (\kappa \Pi^{0,k-1}_K \nabla p_h^{\rmP}, \Pi^{0,k-1}_K \nabla q_h^{\rmP})_{0, K} + S_2^{K} \bigl((I-\Pi^{0,k}_K)p_h^{\rmP}, (I-\Pi^{0,k}_K)q_h^{\rmP} \bigr), }
 \\
\tilde{a}_2^h( p_h^{\rmP}, q_h^{\rmP})|_{K} &:= \tilde{a}_2^{K}( \Pi^{0,k}_{K} p_h^{\rmP},\Pi^{0,k}_{K} q_h^{\rmP}) + S_0^{K} \bigl((I-\Pi^{0,k}_{K})p_h^{\rmP}, (I-\Pi^{0,k}_K)q_h^{\rmP} \bigr),
\end{align*}
where the stabilisations of the bilinear forms $S_1^K(\cdot, \cdot), S_2^K(\cdot, \cdot), S_0^K(\cdot, \cdot)$ acting on the kernel of their respective operators $\bPi^{\beps,k}_K,\;  {\Pi^{0,k}_K},\; \Pi^{0,k}_K$,  are defined as
\begin{align*}
S_1^K(\bu_h,\bv_h) &:= \sigma_1^K\sum_{l=1}^{N^V} \text{dof}_l(\bu_h) \text{dof}_l(\bv_h), \quad \bu_h, \bv_h \in \text{ker}(\bPi^{\beps,k}_K),\\
S_2^K(p_h^{\rmP}, q_h^{\rmP}) &:= \sigma_2^K \sum_{l=1}^{N^Q} \text{dof}_l(p_h^{\rmP}) \text{dof}_l(q_h^{\rmP}), \quad p_h^{\rmP}, q_h^{\rmP} \in \text{ker} ( {\Pi^{0,k}_K}), \\
S_0^K(p_h^{\rmP}, q_h^{\rmP}) &:= \sigma_0^{K} \text{area}(K) \sum_{l=1}^{N^Q} \text{dof}_l(p_h^{\rmP}) \text{dof}_l(q_h^{\rmP}), \quad p_h^{\rmP}, q_h^{\rmP} \in \text{ker} (\Pi^{0,k}_K),
\end{align*}
respectively, where $\sigma_1^K,\sigma_2^K$ and $\sigma_0^K$ are positive multiplicative factors to take into account the magnitude of the physical parameters (and being independent of the mesh size).

Note that  for all $\bv_h \in \bV_h^k(K),\ q_h^{\rmP} \in Q_h^{k,\rmP}(K)$,  we have (see, e.g., \cite{brenner18,daveiga-NS18})
\begin{align} \label{bound:stab}
\begin{split}
\alpha_*a_1^K(\bv_h,\bv_h) &\le S_1^K(\bv_h,\bv_h) \le \alpha^*a_1^K(\bv_h,\bv_h),
\\
\zeta_* a_2^K(q_h^\rmP,q_h^\rmP)& \le S_2^K(q_h^\rmP,q_h^\rmP) \le \zeta^*a_2^K(q_h^\rmP,q_h^\rmP), \\
\tilde{\zeta}_* \tilde{a}_2^K(q_h^\rmP,q_h^\rmP) & \le S_0^K(q_h^\rmP,q_h^\rmP) \le \tilde{\zeta}^*\tilde{a}_2^K(q_h^\rmP,q_h^\rmP),
\end{split}
\end{align}
where $\alpha_*, \alpha^*, \zeta_*, \zeta^*, \tilde{\zeta}_*, \tilde{\zeta}^*$ are positive constants independent of $K$ and $h_K$.
Now, for all $\bu_h, \bv_h \in \bV_h^k,\; p_h^{\rmP},q_h^{\rmP} \in Q_h^{k,\rmP},\;\psi_h, \phi_h\in Z_h^k $, the global discrete bilinear forms are specified as
\begin{gather*}
 a_1^h(\bu_h,\bv_h):= \sum_{K \in \mathcal{T}_h} a_1^h(\bu_h,\bv_h)|_K, \quad
a_2^h(p_h^{\rmP},q_h^{\rmP}):= \sum_{K \in \mathcal{T}_h^\rmP} a_2^h(p_h^{\rmP},q_h^{\rmP})|_K, \quad
 \tilde{a}_2^h(p_h^{\rmP},q_h^{\rmP}):= \sum_{K \in \mathcal{T}_h^\rmP} \tilde{a}_2^h(p_h^{\rmP},q_h^{\rmP})|_K, \\
 b_1(\bv_h, \phi_h) := \sum_{K \in \mathcal{T}_h} b_1^K (\bv_h, \phi_h), \quad a_3(\psi_h, \phi_h) := \sum_{K \in \mathcal{T}_h} a_3^K(\psi_h, \phi_h), \quad b_2(q_h^{\rmP}, \phi_h) := \sum_{K \in \mathcal{T}_h^\rmP} b_2^K(q_h^{\rmP}, \phi_h^{\rmP}).
\end{gather*}

In addition, we observe that
\begin{equation*}
b_2(p_h^{\rmP},\phi_h)= \frac{\alpha}{{\lambda^{\rmP}}} \sum_{K \in \mathcal{T}_h^\rmP} \int_{K} p_h^{\rmP} \phi_h^{\rmP}
=\frac{\alpha}{{\lambda^{\rmP}}}\sum_{K \in \mathcal{T}_h^\rmP}\int_{K} \Pi^{0,k}_K p_h^{\rmP} \phi_h^{\rmP}.
\end{equation*}

On the other hand, the discrete linear functionals, defined on each element $K$, are
\begin{align*}
F^h(\bv_h)|_K:= \int_K  \bb_h(\cdot, t) \cdot \bv_h, \quad \bv_h \in \bV_h^k, \qquad
G^h(q_h^{\rmP})|_K:= \int_K \ell_h^\rmP(\cdot,t) q_h^{\rmP}, \quad q_h^{\rmP} \in Q_h^{k,\rmP},
\end{align*}
where the discrete load and volumetric source are given by
$$\bb_h(\cdot, t)|_K:=  {\bPi^{0,k-2}_K}
\bb (\cdot,t), \quad \ell_h^\rmP(\cdot, t)|_K:= \Pi^{0,k}_K \ell^\rmP(\cdot, t). $$
In view of \eqref{bound:stab}, the discrete bilinear forms $a_1^h(\cdot, \cdot)$, $\tilde{a}_2^h(\cdot, \cdot) $ and $a_2^h(\cdot, \cdot)$ are coercive and bounded in the following manner \cite{daveiga-b13,vb2016,burger_acom21}
\begin{gather*}
a_1^h(\bu_h, \bu_h) \ge 2 \, {\mu_{\min}} \min \lbrace 1, \alpha_* \rbrace\,\, \| \beps(\bu_h) \|_0^2,  \qquad
a_2^h(q_h^{\rmP}, q_h^{\rmP})  \ge \min \lbrace 1, \zeta_* \rbrace\, \frac{\kappa_{\min}}{\eta} \, \| \nabla q_h^{\rmP} \|_{0, \OmP}^2, \\
a_2^h(p_h^{\rmP}, q_h^{\rmP})  \le \max \lbrace 1, \zeta^* \rbrace\, \frac{\kappa_{\max}}{\eta} \, \| \nabla p_h^{\rmP} \|_{0, \OmP} \| \nabla q_h^{\rmP} \|_{0, \OmP}, \quad 
a_1^h(\bu_h, \bv_h)  \le 2\, {\mu_{\max}} \max \lbrace 1, \alpha^* \rbrace\, \| \beps(\bu_h) \|_0 \| \beps (\bv_h) \|_0,\\
\tilde{a}_2^h(q_h^{\rmP}, q_h^{\rmP})  \ge \min \lbrace 1, \tilde{\zeta}_* \rbrace\, \Big( c_0 +  \frac{\alpha^2}{{\lambda^{\rmP}}} \Big) \, \| q_h^{\rmP} \|_{0, \OmP}^2, \qquad
\tilde{a}_2^h(p_h^{\rmP}, q_h^{\rmP}) \le \max \lbrace 1, \tilde{\zeta}^* \rbrace\, \Big( c_0 +  \frac{\alpha^2}{{\lambda^{\rmP}}} \Big)\, \| p_h^{\rmP} \|_{0, \OmP} \| q_h^{\rmP} \|_{0, \OmP},
\end{gather*}
for all $\bu_h,\bv_h \in \bV_h^k$, $p_h^{\rmP},q_h^{\rmP} \in Q_h^{k,\rmP}${, where $\mu_{\min}:=\min\{\mu^\mathrm{E},\mu^\mathrm{P}\}, \mu_{\max}:=\max\{\mu^\mathrm{E},\mu^\mathrm{P}\}$}.
Moreover, from the definitions of the operators~$\bPi^{0,k}_K$ and~$\Pi^{0,k}_K$,
we may deduce that the following bounds hold for the linear functionals:
\begin{equation*}
F^h(\bv_h)  \le \| \bb\|_0 \|\bv_h\|_0, \qquad G^h(q_h) \le \| \ell^\rmP \|_{0, \OmP} \| q_h^{\rmP} \|_{0, \OmP},
  \qquad \text{for all  $\bv_h \in \bV_h^k$, $q_h^{\rmP} \in Q_h^{k,\rmP}$}.
\end{equation*}

We also recall that the bilinear form $b_1(\cdot,\cdot)$ satisfies the following discrete inf-sup condition on $\bV_h^k\times Z_h^k$: there exists $\tilde{\beta}>0$, independent of $h$, such that (see \cite{daveiga-NS18}),
\begin{equation}\label{discr-infsup}
\sup_{\bv_h (\neq 0) \in \bV_h^k} \frac{b_1(\bv_h, \phi_h)}{\| \bv_h \|_1}
\ge \tilde{\beta} \| \phi_h \|_0 \quad \text{for all $\phi_h \in Z_h^k$.}
\end{equation}

The semidiscrete VE formulation is now defined as follows: For all $t>0$, given $\bu_h(0)$, $p_h^{\rmP}(0)$, $\psi_h(0)$, find $\bu_h \in \bL^2((0,t_{\text{final}}],\bV_h^k)$, $ p_h^{\rmP} \in L^2((0,t_{\text{final}}],Q_h^{k,\rmP}), \;\psi_h \in L^2((0,t_{\text{final}}],Z_h^k)$ with $ \partial_t p_h^{\rmP} \in L^2((0,t_{\text{final}}], Q_h^{k,\rmP})$, $\partial_t \psi_h \in L^2((0,t_{\text{final}}],Z_h^k)$ such that
\begin{subequations} \label{eq3.7}
	\begin{alignat}{5}
	&&   a_1^h(\bu_h,\bv_h)   &&                 &\;+&\; b_1(\bv_h,\psi_h)     &=&\;F^h(\bv_h)&\quad\forall \bv_h \in \bV_h^k, \label{weak-uh}\\
	\tilde{a}_2^h(\partial_t p_h^{\rmP},q_h^{\rmP}) &\; +&               &&      a_2^h(p_h^{\rmP},q_h^{\rmP})   &\;-&\;   b_2( q_h^{\rmP}, \partial_t \psi_h)  &=&\;G^h(q_h^{\rmP}) &\quad\forall q_h^{\rmP} \in Q_h^{k,\rmP}, \label{weak-ph}\\
	&&b_1(\bu_h,\phi_h)  &\;+\;& b_2(p_h^{\rmP},\phi_h)&\;-&\; a_3(\psi_h,\phi_h) &=&0 &\quad\forall\phi_h \in Z_h^k. \label{weak-psih}
	\end{alignat}
\end{subequations}
The following result will be used for   the stability and  error estimates for the semi-discrete scheme without employing the Gronwall's inequality. For a detailed proof, we refer to \cite{burger_acom21} and \cite[Lemma 3.2]{lee19}.

\begin{lemma} \label{lem:semi-Xbound}
	Let $X(t)$ be a continuous function,  and consider the non-negative functions $F(t)$ and $D(t)$ satisfying, for constants $C_0 \ge 1$ and $C_1 > 0$, the bound
	\begin{align*}
	X^2(t) \le C_0 X^2(0) + C_1 X(0) + D(t) + \int_0^t F(s) X(s) \, \mathrm{d}s, \quad \forall ~ t \in [0,t_{\mathrm{final}}].
	\end{align*}	
	Then, for each $t \in [0,t_{\mathrm{final}}]$, there holds
	\begin{align}
	\label{semi-Xbound}
	X(t) \lesssim X(0) + \max \left\lbrace  C_1 + \int_0^t F(s) \, \mathrm{d}s,~ D(t)^{1/2} \right\rbrace.
	\end{align}
\end{lemma}
Squaring both sides of \eqref{semi-Xbound} and using  Cauchy--Schwarz inequality, we can rewrite \eqref{semi-Xbound} as
\begin{align}
\label{new:semi-Xbound}
X(t)^2 \lesssim X(0)^2 + \max \left\lbrace  C_1^2 + \int_0^t F(s)^2 \, \mathrm{d}s,~ D(t) \right\rbrace.
\end{align}

Now we establish the stability of \eqref{eq3.7}.

\begin{theorem}[Stability of the semi-discrete problem]\label{thm:stab-semi}
	Let $(\bu_h(t), p_h^{\rmP}(t), \psi_h(t))$ solve
	\eqref{eq3.7}  for each $t \in (0,t_{\mathrm{final}}]$. Then  there exists a constant $C>0$ independent of {$c_0, \lambda$,} and $h$,  such that
	\begin{align} \label{bound:semi-stability} \begin{split}
	& {\mu_{\min}} \| \beps(\bu_h(t)) \|_0^2 + c_0 \| p_h^{\rmP}(t) \|_{0,\OmP}^2  + \| \psi_h(t) \|_0^2 + \frac{\kappa_{\min}}{\eta} \int_0^t\| \nabla p_h^{\rmP}(s) \|_{0,\OmP}^2 \, \mathrm{d}s  \\
	&   \le C \bigg( {\mu_{\min}} \| \beps(\bu_h(0))\|_0^2 + {\Big(c_0 + \frac{\alpha^2}{{\lambda^{\rmP}}} \Big) } \| p_h^{\rmP}(0) \|_{0,\OmP}^2 + {\frac{1}{{\lambda_{\min}}}} \| \psi_h(0) \|_0^2 +
	\int_0^t \| \partial_t \bb(s)\|_0^2 \, \mathrm{d}s \\
	& \qquad \quad
	+ \sup_{t \in [0,t_{\mathrm{final}}] }  \| \bb(t)\|_0^2 + \int_0^t \| \ell^{\rmP}(s) \|_{0,\OmP}^2 \, \mathrm{d}s \bigg).
	\end{split} \end{align}	
\end{theorem}
The proof is directly followed by   using the similar arguments used in the proof given in  \cite[Theorem 3.1]{burger_acom21} and Lemma \ref{lem:semi-Xbound}. Therefore, we skip the proof. Next,  by using \eqref{bound:semi-stability}, we established the well-posedness of the semi-discrete scheme.

\begin{corollary}[Solvability of the discrete problem]\label{coer1}
	Problem \eqref{eq3.7} has a unique solution in $\bV_h^k \times Q_h^{k,\rmP} \times Z_h^k$ for a.e. $t \in (0,t_{\mathrm{final}}]$.
\end{corollary}
\begin{proof}	
	Let  $\bu_h(t) := \sum_{i=1}^{N^{\bV}} \mathcal{U}_i(t) \xi_i$, $p_h^{\rmP}(t) := \sum_{j=1}^{N^Q} \mathcal{P}_j(t) \chi_j^{\rmP}$, $\psi_h(t) := \sum_{l=1}^{N^Z} \mathcal{Z}_l(t) \Phi_l$ where $\xi_i ( 1 \le i \le N^{\bV})$, $\chi_j^{\rmP} (1 \le j \le N^Q)$, $\Phi_l ( 1 \le l \le N^Z$, where $N^Z$ coincides with the number of elements in $\mathcal{T}_h$) are the basis functions for the spaces $\bV_h^k, Q_h^{k,\rmP}, Z_h^k$ respectively. Then  \eqref{eq3.7} can be written as
	\begin{align} \label{eq:matrixform}
	 \mathcal{A}
	\begin{pmatrix}
	\dot{\mathcal{U}}(t) \\ \dot{\mathcal{P}}(t) \\ \dot{\mathcal{Z}}(t)
	\end{pmatrix}
	+
	 \mathcal{B}
	\begin{pmatrix}
	\mathcal{U}(t) \\ \mathcal{P}(t) \\ \mathcal{Z}(t)
	\end{pmatrix}
	= \begin{pmatrix}
	\boldsymbol{F}(t) \\ \boldsymbol{G}(t) \\ \boldsymbol{0}
	\end{pmatrix}.
	\end{align}
System \eqref{eq:matrixform} possesses a unique solution if the matrix $\mathcal{A+B}$ is invertible. We then consider: For $(L_1^h, L_2^h, L_3^h) \in (\bV_h^k \times Q_h^{k,\rmP} \times Z_h^k)'$, find $\bu_h \in \bV_h^k$, $p_h^{\rmP}~\in~Q_h^{k,\rmP}$, $\psi_h \in Z_h^k$ such that
	\begin{subequations}\label{weak-new}
		\begin{alignat}{5}
		&&   a_1^h(\bu_h,\bv_h)   &&                 &\;+&\; b_1(\bv_h,\psi_h)     &=&\;L_1^h(\bv_h)&\quad\forall \bv_h \in \bV_h^k, \label{weak-uh-new}\\
		\tilde{a}_2^h( p_h^{\rmP},q_h^{\rmP}) &\; +&               &&      a_2^h(p_h^{\rmP},q_h^{\rmP})   &\;-&\;   b_2( q_h^{\rmP},  \psi_h)  &=&\;L_2^h(q_h^{\rmP}) &\quad\forall q_h^{\rmP} \in Q_h^{k,\rmP}, \label{weak-ph-new}\\
		&&b_1(\bu_h,\phi_h)  &\;+\;& b_2(p_h^{\rmP},\phi_h)&\;-&\; a_3(\psi_h,\phi_h) &=& L_3^h(\phi_h) &\quad\forall\phi_h \in Z_h^k. \label{weak-psih-new}
		\end{alignat}\end{subequations}
The unique solvability of \eqref{weak-new} (and the invertibility of  $\mathcal{A+B}$) is established   showing that the homogenous form of  \eqref{weak-new} has only the trivial solution. Setting  $L_1^h(\bv_h)=L_2^h(q_h^{\rmP})=L_3^h(\phi_h)=0$, and  choosing  $\bv_h= \bu_h, \phi_h=\psi_h, q_h^{\rmP}=p_h^{\rmP}$ in  \eqref{weak-new}, we  obtain the following bound  by proceeding   as in the proof of \eqref{bound:semi-stability}
	\begin{align*}
	{\mu_{\min}} \| \beps(\bu_h) \|_0^2 + c_0 \| p_h^{\rmP} \|_{0, \OmP}^2  +  \frac{\kappa_{\min}}{\eta} \| \nabla p_h^{\rmP} \|_{0, \OmP}^2 \le 0,
	\end{align*}
	and hence an application of the Poincar\'{e} and Korn inequalities together with the inf-sup condition of $b_1(\cdot, \cdot)$ yields $\bu_h=\cero$, $p_h^{\rmP}=0$ and $\psi_h=0$.
\end{proof}

Next, we discretise in time using the backward Euler method with  constant step size $\Delta t = t_{\mathrm{final}} / N$ and denote any function $f$ at $t=t_n$ by $f^n$. The fully discrete scheme reads: Given $\bu_h^0$, $p_h^{0,P}$, $\psi_h^0$, and for $t_n=n\Delta t$, $n=1, \dots, N$, find $\bu_h^{n} \in \bV_h^k$, $p_h^{n, \rmP} \in Q_h^{k,\rmP}$ and $\psi_h^{n} \in Z_h^k$ such that
\begin{subequations}  \label{fd-scheme}
	\begin{align}
		a_1^h(\bu_h^{n},\bv_h)  + b_1(\bv_h,\psi_h^{n})    & = F^{h,n}(\bv_h), \label{weak-uhn}\\
		\tilde{a}_2^h \left( p_h^{n,\rmP} , q_h^{\rmP} \right) + \Delta t a_2^h( p_h^{n,\rmP},q_h^{\rmP}) - b_2 \left( q_h^{\rmP}, \psi_h^{n} \right)  
	 &
	= \Delta t G^{h,n} (q_h^{\rmP})+\tilde{a}_2^h \left(  p_h^{n-1, \rmP} , q_h^{\rmP} \right) -b_2 \left( q_h^{\rmP}, \psi_h^{n-1} \right), \label{weak-phn}\\
		b_1(\bu_h^{n},\phi_h)  + b_2(p_h^{n, \rmP},\phi_h) - a_3(\psi_h^{n},\phi_h) & = 0, \label{weak-psihn}
	\end{align}
\end{subequations} 	
for all $\bv_h \in \bV_h^k,q_h^{\rmP} \in Q_h^{k,\rmP},\phi_h \in Z_h^k$;  where
\begin{align*}
F^{h,n}(\bv_h)|_K:= \int_K  \bb_h(t^n) \cdot \bv_h, \quad
G^{h,n}(q_h^{\rmP})|_K:= \int_{K \in \cT_h^{\rmP}} \ell_h^{\rmP}(t^n) q_h^{\rmP}.
\end{align*}
Next we provide the following auxiliary result proved in \cite{burger_acom21}.
\begin{lemma} \label{lem:Xn-bound}
	Let $X_n, ~ 1 \le n \le N$ be a finite sequence of functions with non-negative constants $C_0, C_1$ and finite sequences~$D_n$
	and~$G_n$ such that
$	X_n^2 \le C_0 X_0^2 + C_1 X_0 + D_n + \sum_{j=1}^n G_j X_j$, for all $1 \le n \le N$. 	Then
	\[X_n^2 \lesssim X_0^2 + \max \left\lbrace C_1^2 + \sum_{j=1}^n G_j^2,~ D_n  \right\rbrace \quad  \text{for all $1 \le n \le N$.}\]
\end{lemma}
Employing Lemma \ref{lem:Xn-bound} and following analogously to the proof of the Theorem 3.2 of \cite{burger_acom21}, it is easy to show the proposed fully discrete scheme \eqref{fd-scheme} is stable, and we formally state the result below.    
\begin{theorem}[Stability of the fully-discrete problem]\label{thm:stab-fully}
	The unique solution to  \eqref{fd-scheme} depends   continuously  on data.
	More precisely, there exists a constant~$C$ independent of ${c_0}, \lambda, h$ and  $\Delta t$ such that
	\begin{align*}
	& {\mu_{\min}} \| \beps (\bu_h^n) \|_0^2 + c_0 \| p_h^{n,\rmP} \|_{0, \OmP}^2 + \| \psi_h^n \|_0^2  + (\Delta t)\frac{\kappa_{\min}}{\eta} \sum_{j=1}^n \| \nabla p_h^{j, \rmP} \|_{0, \OmP}^2 \nonumber\\
	& \le C \bigg( {\mu_{\min}} \| \beps (\bu_h^0)\|_0^2 + {\Big(c_0 + \frac{\alpha^2}{{\lambda^{\rmP}}} \Big)} \| p_h^{0, \rmP} \|_{0, \OmP}^2 + \| \psi_h^0 \|_0^2 + \max_{0 \le j \le n} \| \bb^j \|_0^2 
	\\
	& \qquad \quad
	+ (\Delta t) \sum_{j=1}^n \Big( \| \partial_{t} \bb^j \|_0^2 + \|\ell^{j, \rmP}\|_{0, \OmP}^2 \Big) + (\Delta t)^2 \int_{0}^{t_{\mathrm{final}}} \| \partial_{tt} \bb(s) \|_0^2 \, \mathrm{d}s \bigg),\nonumber
	\end{align*}
	with $\bb^m:=\bb(\cdot,t^m)$ and $\ell^{m,\rmP}:=\ell^{\rmP}(\cdot,t^m)$, for $m=1,\ldots,n$.
\end{theorem}
We can establish the following result by proceeding on the same  lines as the proof of Corollary \ref{coer1}.

\begin{corollary}[Solvability of the fully discrete problem]
	The problem \eqref{fd-scheme} has a unique solution in $\bV_h^k \times Q_h^{k,\rmP} \times Z_h^k$.
\end{corollary}

\subsection{Three dimensional virtual element space}\label{subsec:3D}
Likewise to the two-dimensional decomposition of domain $\Omega$, we discretise each codomain $\OmP$ and $\OmE$ independently into polyhedron elements $K$ having faces $\mathit f$, as $\cT_h^\rmP$ and $\cT_h^\rmE$ respectively then the unification of such polyhedrons on whole domain is denoted as $\cT_h$. By $N^{\mathit f}_K$ we will denote the number of faces $\mathit f$ in the polyhedron $K$. Now, we mention below the sufficient assumptions for obtaining the error analysis in addition with the earlier stated assumptions ($A$) and ($\tilde{A}$) (for two-dimension):
\begin{itemize}
	\item[($A_1$)] every face $\mathit f$ of $K$ is star-shaped within a ball of radius $\tilde{C}_{{\mathcal T}}h_{\mathit f}$;
	\item[($\tilde{A}_1$)] the ratio between shortest edge $h_e$ and $h_{\mathit f}$ as well as $h_{\mathit f }$ and $h_K$ is larger than $\tilde{C}_{\mathcal T}$, that is
	$$h_e \ge \tilde{C}_{\mathcal T} h_{\mathit f} \ge \tilde{C}_{\mathcal T}^2 h_K.$$
\end{itemize}
Note that  assumption ($\tilde{A}_1$) implies  assumption ($\tilde{A}$), and hence it is not required to add it separately.

Now, we define the local three dimensional VE spaces, introduced in \cite{ahmad13,daveiga-div17}, for the displacement on each element $K \in \cT_h$, and fluid pressure on each element $K \in \cT_h^\rmP$ as
\begin{align*} 
\bV_h^k(K) & := { \Big\{ \bv \in [H^1( K) \cap C^0(\partial K)]^3:
\begin{cases}
- \bDelta \bv + \nabla s \in \mathcal{G}_{k-2}^{\perp}(K) ~ \text{ for some } s \in L^2(K) \\
\vdiv \bv \in \mathbb{P}_{k-1}(K), \quad \bv |_f \in \bW_k(f) \,\, \forall f \in \partial K
\end{cases} \Big\}, } \\
Q_h^{k,\rmP}(K) & := \bigl\{ q^{\rmP} \in H^1( K) \cap C^0(\partial K): \Delta q^{\rmP}|_{K} \in \mathbb{P}_k(K), \quad q^{\rmP} |_{f} \in Q_h^{k,\rmP}(f)  \,\, \forall f \in \partial K, \\
&  \hspace{4.6cm} (\Pi_{K}^{\nabla,k} q^{\rmP} - q^{\rmP}, m_{\alpha})_{0,K} = 0\ \forall m_{\alpha} \in \mathcal{M}_k \backslash\mathcal{M}_{k-2}(K) \bigr\},
	\end{align*}
where the element boundary spaces
$Q_h^{k,\rmP}(\mathit f)$ taken from definition \eqref{VE-spaces} for two-dimensional element, and $\bW_h^k(\mathit f)$ on each face $\mathit f$ is given as
\begin{align*}
\bW_h^k(\mathit f) & := \bigl\{ \bw \in [H^1( \mathit f)]^3: \Delta \bw|_{\mathit f} \in [\mathbb{P}_k(\mathit f)]^3, \, \bw |_{\partial \mathit f} \in [\mathbb{P}_k(\partial \mathit f)]^3, \\ & \quad \quad  \qquad\quad \qquad \quad  \quad \; (\Pi_{\mathit f}^{\nabla,k} \bw - \bw, m_{\alpha})_{0,\mathit f} = 0\ \forall m_{\alpha} \in [\mathcal{M}_k \backslash\mathcal{M}_{k-2}(\mathit f)]^3 \bigr\}.
\end{align*}

The degrees of freedom for the local displacement space $\bV_h^k(K)$, consist of
($D_v1$)-($D_v4$), ($D_v5$) the moments:
$ \int_{\mathit f} \bv_h \cdot {\mathit m}_{\alpha}, \quad \forall {\mathit m}_{\alpha} \in [\mathcal{M}_{k-2}(\mathit f)]^3$. Then for the local fluid pressure space $Q_h^{k,\rmP}(K)$, we have: ($D_q1$)-($D_q3$), ($D_q4$) the moments of $q^{\rmP}$:
$\int_{\mathit f} q^{\rmP}~ m_{\alpha}, \quad \forall m_{\alpha} \in \mathcal{M}_{k-2}(\mathit f).$

Note that the local space $Z_h^k(K)$ for the global-pressure remains same, given in \eqref{VE-spaces}. Therefore the degrees of freedom for the local global-pressure space $Z_h^k(K)$ is ($D_z$).

The dimension of $\mathbb{P}_{k-2}(K)$ is $\frac{(k-1)k(k+1)}{6}$ and  $\mathcal{G}^{\perp}_{k-2}(K)$ is $\frac{(k-2)k(k+1)}{2} + 1$ implying that the dimension of $\bV_h^k(K)$ is $3 k N^v_K + \frac{3 k(k-1)}{2} N^f_K +  \frac{(k-2)k(k+1)}{2} + \frac{k(k+1)(k+2)}{6}$; the dimension of $Q_h^{k,\rmP}(K)$ is $k N^v_{K}  + \frac{k(k-1)}{2} N^f_K + \frac{k(k^2-1)}{6}$, and that of $Z_h^k(K)$ is $\frac{k(k+1)(k+2)}{6}$.
Then the global VE spaces are constructed keeping the same description for global spaces given in Section \ref{subsec:VEspaces}. Likewise, the degrees of freedom for these global spaces are combination of local degrees of freedom over all elements $K$ except the ones containing domain boundary vertices and nodes on domain boundary faces. The computation of these degrees of freedom is detailed in  \cite{ahmad13,daveiga-div17} and the analysis is outlined in the following section. 

\section{A priori error estimates} \label{sec:estimates}
For the sake of error analysis, we require additional regularity:
In particular, for any $t>0$, we consider that the global displacement
is $\bu(t) \in {[H^{k+1}(\OmP\cup\OmE)]^d}$,
the fluid pressure $ p^{\rmP}(t)\in H^{k+1}(\OmP)$, and the  total and elastic pressures
$\psi^{\rmP}(t) \in H^{k}(\OmP),~\psi^{\rmE}(t) \in H^{k}(\OmE)$. Furthermore,
our subsequent analysis also requires the following regularity in time:
$\partial_{t} \bu \in \bL^2(0,T; [H^{k+1}(\OmP\cup\OmE)]^d)  $, $\partial_{t} p^{\rmP} \in L^2(0,T; H^{k+1}(\OmP))$, $\partial_{t} \psi^{\rmP} \in L^2(0,T; H^k(\OmP))$, $\partial_{t} \psi^{\rmE} \in L^2(0,T; H^k(\OmE))$, $\partial_{tt} \bu \in \bL^2(0,T; [L^2(\Omega)]^d)$ and $\partial_{tt} p^{\rmP}, \partial_{tt} \psi \in L^2(0,T;L^2(\Omega))$.

We start by recalling an
estimate for the interpolant $\bu_I \in \bV_h^k$ of $\bu$ and $p_I^{\rmP} \in Q_h^{k, \rmP}$ of $p^{\rmP}$ (see \cite{antonietti14,CGPS,CMS2016}).
\begin{lemma} \label{interpolant_u}
	There exist interpolants $\bu_I \in \bV_h^k$  and  $p_I^{\rmP} \in Q_h^{k, \rmP}$ of~$\bu$ and~$p^{\rmP}$, respectively,
	such that
	\begin{align*}
	\| \bu - \bu_I \|_0 + h | \bu - \bu_I|_1 &\lesssim h^{k+1}
 {\big( |\bu^{\rmP}|_{k+1,\OmP} + |\bu^{\rmE}|_{k+1,\OmE} \big),} \\
	\| p^{\rmP} - p_I^{\rmP} \|_{0, \OmP} + h | p^{\rmP} - p_I^{\rmP}|_{1, \OmP} &\lesssim h^{k+1} |p^{\rmP}|_{k+1, \OmP}.
	\end{align*}
\end{lemma}

We now introduce the poroelastic projection operator: given
$(\bu,p^{\rmP},\psi)\in \bV\times Q^{\rmP} \times Z$, find $I^h(\bu,p^{\rmP},\psi) :=( I^h_{\bu} \bu, I^h_p p^{\rmP}, I^h_{\psi} \psi)$
$\in \bV_h^k \times Q_h^{k, \rmP} \times Z_h^k$ such that
\begin{subequations}
	\begin{alignat}{5}
	a_1^h(I_{\bu}^h \bu,\bv_h) &\;+&\; b_1(\bv_h, I_{\psi}^h \psi) & = &a_1(\bu,\bv_h)  &\;+&\; b_1(\bv_h, \psi) &\quad\text{for all $\bv_h \in \bV_h^k$,}  \label{weak-Iuh}\\
	& \;&\; b_1(I_{\bu}^h \bu, \phi_h)  &=& b_1(\bu, \phi_h)   & \;&\;  &\quad \text{for all $\phi_h \in Z_h^k$,}  \label{weak-Ipsih} \\
	& \;&\; a_2^h(I^h_p p^{\rmP},q_h^{\rmP})  &=& a_2(p^{\rmP},q_h^{\rmP})   & \;&\;  &\quad \text{for all $q_h^{\rmP} \in Q_h^{k,\rmP}$,}  \label{weak-Iph}
	\end{alignat}\end{subequations}
and we remark that $I^h$ is defined by the combination of the
saddle-point problem \eqref{weak-Iuh}, \eqref{weak-Ipsih} and the
elliptic problem \eqref{weak-Iph}; and hence, it is well-defined.
\begin{lemma}[Estimates for the poroelastic  projection]
	Let $(\bu, p^{\rmP}, \psi)$ and $( I^h_{\bu} \bu, I^h_p p^{\rmP}, I^h_{\psi} \psi)$ be the unique
	solutions of \eqref{weak-uh}--\eqref{weak-psih} and \eqref{weak-Iuh}, \eqref{weak-Ipsih},
	respectively. Then the following estimates hold:
	\begin{subequations}\begin{align}
		\| \bu - I^h_{\bu} \bu \|_1 + \| \psi - I^h_{\psi} \psi \|_0  &\lesssim h^k ({ |\bu^{\rmP}|_{k+1,\OmP} + |\bu^{\rmE}|_{k+1,\OmE}} + |\psi^{\rmP}|_{k, \OmP} +|\psi^{\rmE}|_{k, \OmE}), \label{estimate-Ihu-psi} \\
		\| p^{\rmP} - I^h_p p^{\rmP} \|_{0, \OmP} + h | p^{\rmP} - I^h_p p^{\rmP} |_{1, \OmP} &\lesssim h^{k+1} |p^{\rmP}|_{k+1, \OmP}  \label{estimate-Ihp}.
		\end{align}\end{subequations}
\end{lemma}
\begin{proof}
It follows from estimates available for discretisations of Stokes \cite{antonietti14} and elliptic problems \cite{daveiga-g13}.
\end{proof}

Note that the arguments here are readily extendible to
 derive error estimates of order $h^s$. It suffices to assume that
	$\bu(t) \in {[H^{1+s}(\OmP\cup\OmE)]^d}$,  $ p^{\rmP}(t)\in H^{1+s}(\OmP)$, $\psi^{\rmP}(t) \in H^{s}(\OmP)$, and $\psi^{\rmE}(t) \in H^{s}(\OmE)$, for $0 < s \le k$.


\begin{theorem}[Semi-discrete energy error estimates]\label{th:semid}
	Let the triplets $(\bu(t),p^{\rmP}(t),\psi(t)) \in \bV \times Q^{\rmP} \times Z$
	and $(\bu_h(t),p_h^{\rmP}(t),\psi_h(t)) \in \bV_h^k \times Q_h^{k,\rmP} \times Z_h^k$ be the unique
	solutions to problems \eqref{weak-u}--\eqref{weak-psi} and \eqref{weak-uh}--\eqref{weak-psih},
	respectively. Then, the following bound holds, with constant $C>0$ independent of~$h$, $c_0$ and $\lambda$,
	\begin{align*}
	{\mu_{\min}} \| \beps((\bu- \bu_h)(t)) \|_0^2  + \|(\psi - \psi_h)(t) \|_0^2  
	+ \frac{\kappa_{\min}}{\eta} \int_0^t \| \nabla (p^{\rmP} - p_h^{\rmP})(s) \|_{0, \OmP}^2 \, \mathrm{d}s \,\le \, C\, h^{2k}.
	\end{align*}
\end{theorem}
 The details of the proof are given in Appendix \ref{proof-semi}.

\begin{theorem}[Fully-discrete error estimates]\label{thm:fullyd}
	Let ($\bu(t),p^{\rmP}(t),\psi(t)$) $\in \bV \times Q^{\rmP} \times Z$
	and ($\bu_h^n,p_h^{n,\rmP},\psi_h^n$) $\in \bV_h^k \times Q_h^{k,\rmP} \times Z_h^k$ be the unique
	solutions to problems \eqref{weak-u}-\eqref{weak-psi} and \eqref{weak-uhn}-\eqref{weak-psihn},
	respectively. Then the following estimate holds for any $n=1,\ldots,N$, with constants~$C$ independent of~$h,\, \Delta t ,\, \lambda$  {and $c_0$}:
	\begin{align} \label{th4.3-est}
	{\mu_{\min}} \| \beps(\bu(t_n) - \bu_h^n) \|_0^2 \, + \, \| \psi(t_n) - \psi_h^n \|_0^2  
	+ (\Delta t) \frac{\kappa_{\min}}{\eta}\| \nabla (p^{\rmP}(t_n) - p_h^{n,\rmP})\|_{0, \OmP}^2   \, \le C\, (h^{2 k} + \Delta t^2).
	\end{align}
\end{theorem}
The proof of the above theorem is postponed to Appendix \ref{proof-full}.

	\section{Computational results} \label{sec:numer}
In this section we collect numerical tests that illustrate the convergence of the proposed VEMs. The implementation is based on an in-house MATLAB library. Errors between exact solutions (evaluated at integration points) and projection of VE solutions will be measured in the computable proxies for the absolute $L^2(\Omega)$--norm and $H^1(\Omega)$--semi-norm, which, for an approximation space of order $k$, are defined as follows 
%
	\begin{align*}
	E_0^k(v) := \| v - \Pi^{0,k}_K v_h\|_{0,\Omega} \qand E_1^k(v) := \| \nabla v - \nabla \Pi^{\nabla,k}_K v_h\|_{0,\Omega},
	\end{align*}
	respectively, 
	where the generic field $v$ can be global displacement $\bu$,
	fluid pressure $p^{\rmP}$, or the global pressure $\psi$ (and with the subscript $h$ we denote the approximate field) and $v_h$ denotes the discrete solution by the $k$-order VE scheme.  
	The experimental decay  rates associated with the errors $E_j^k(v) $ and $\widetilde{E}_j^k(v) $ (using approximation spaces of order $k$) generated by the VEM on meshes with maximal size 
 $h$ and $\widetilde{h}$, respectively, are denoted as 
	\[ r_j^k(v) = \frac{\log(E_j^k(v) /\widetilde{E}_j^k(v) )}{\log(h/\widetilde{h})} \qquad j = 0,1.\]
In the case of stationary solutions, according to Theorem \ref{thm:fullyd} we expect that these errors decay with $\mathcal{O}(h^{k})$. 
	
	\begin{figure}[t!]
		\centering
		\includegraphics[width=0.25\linewidth]{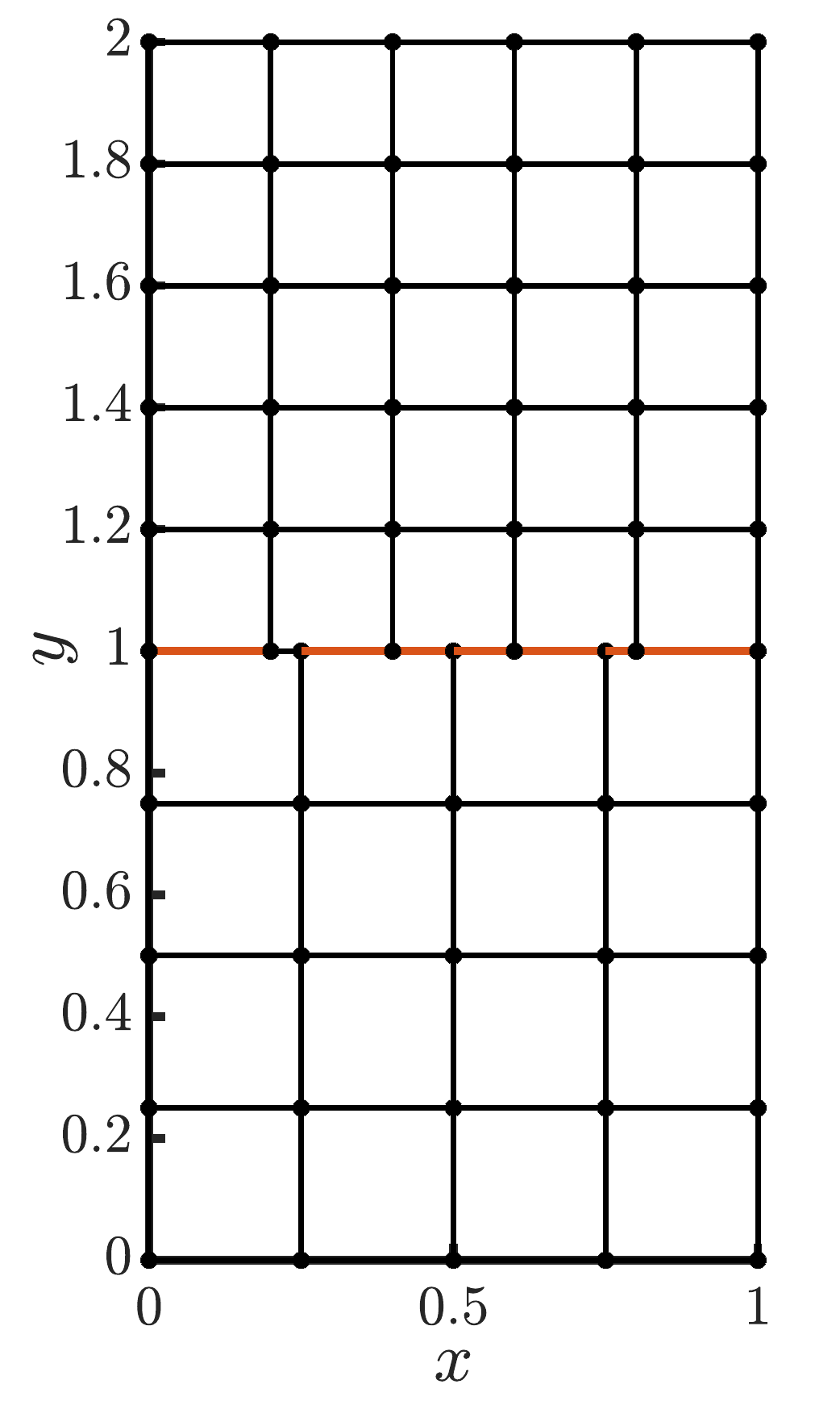}
		\caption{Square mesh with smaller edges at interface on domain $\Omega_1$.}
		\label{fig:sq_mesh}
	\end{figure}	
	
We will use polygonal meshes of different type, which are allowed to possess small edges. The computational results will include 
 tests produced using either the classical \emph{dofi-dofi}
stabilisation (defined in Section~\ref{sec:VEapprox}) as well as the tangential edge stabilisation
proposed in \cite{daveiga17} for elliptic problems, and here adapted as
	 	\begin{equation*}
	 	S^K_{\partial} (\bu, \bv) := 
	 	\sum_{e\in \partial K}  h_{K} \int_{ \partial K} [ \bnabla \bu\, \bt_K^e ] \cdot  [ \bnabla \bv\, \bt_K^e ] .
	 	\end{equation*}

\begin{figure}[t!]
	\begin{center}
	\subfigure[]{\includegraphics[width=0.325\textwidth]{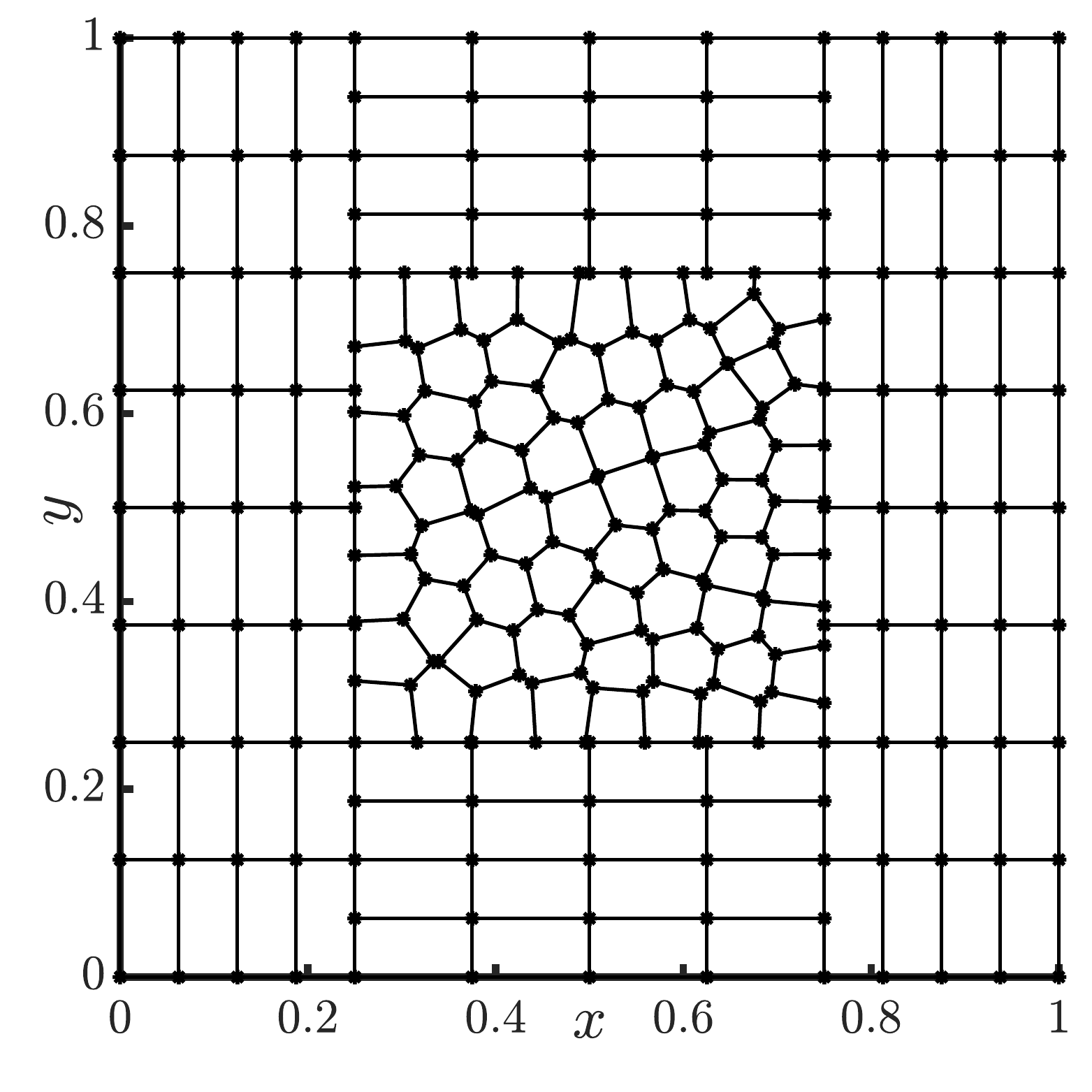}}
	\subfigure[]{\includegraphics[width=0.325\textwidth]{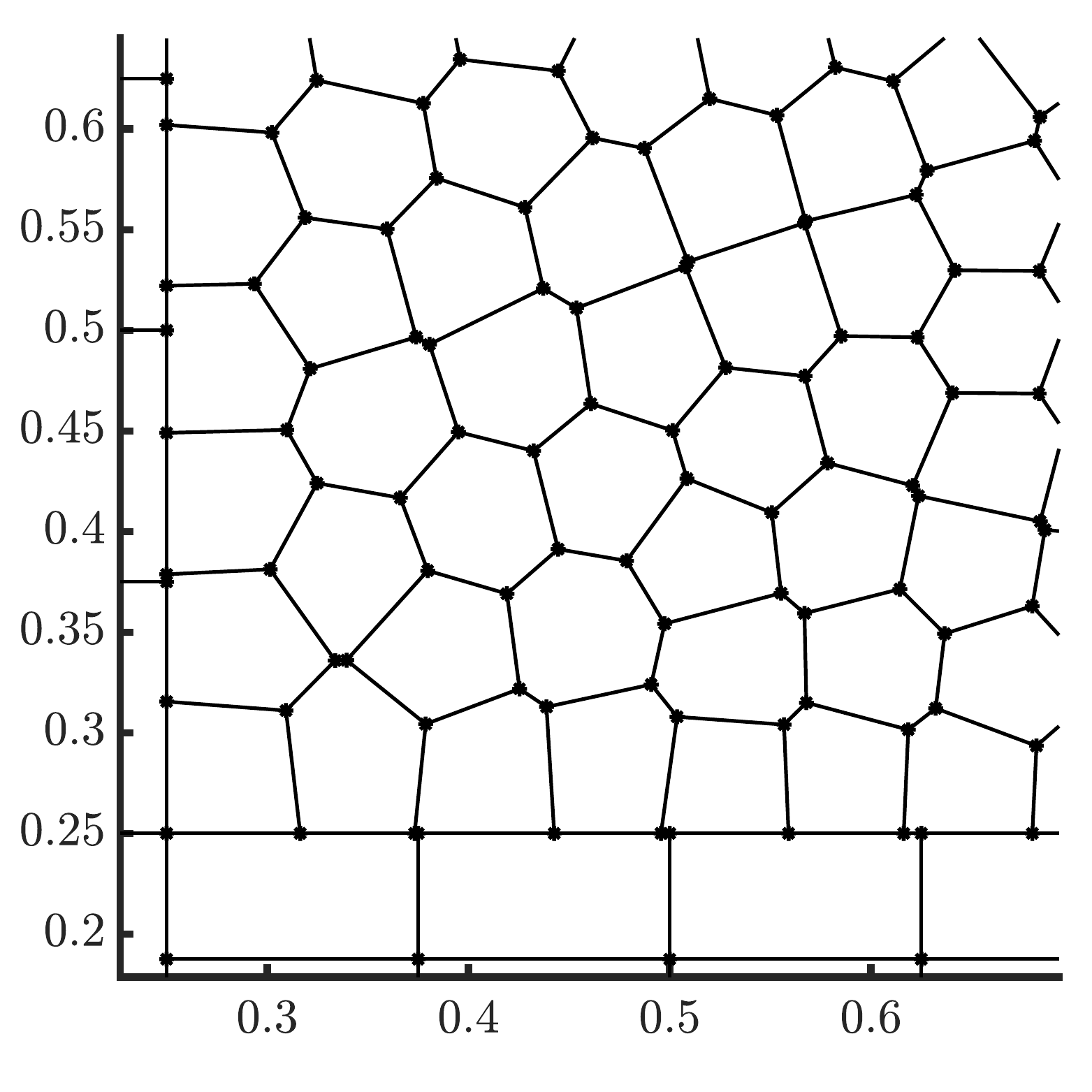}}
	\subfigure[]{\includegraphics[width=0.325\textwidth]{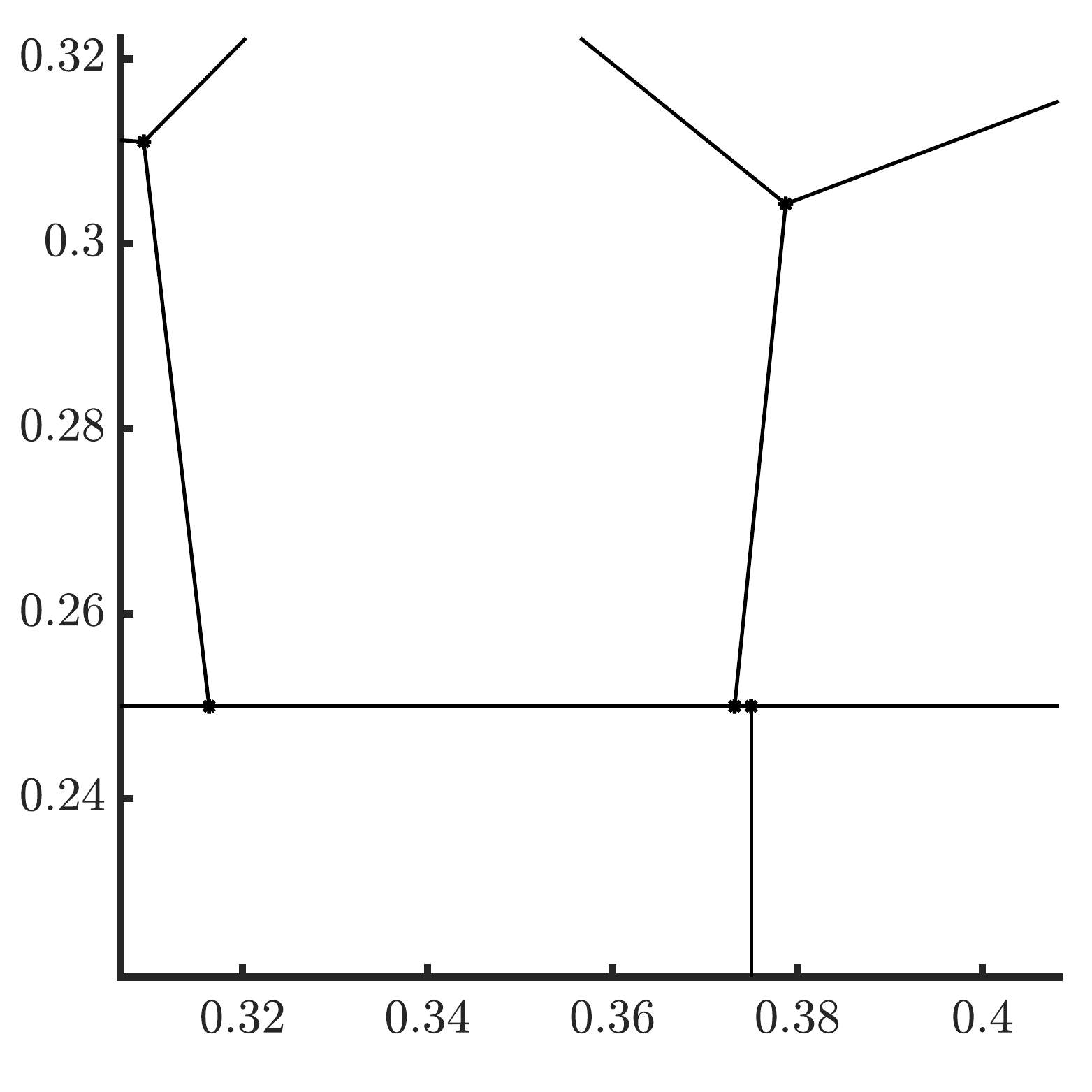}}
	\end{center}
	
	\vspace{-4mm}
	\caption{Polygonal mesh with small edges at interface on domain $\Omega$, and zoomed images near the point $(x,y)=(0.375,0.25)$.}
	\label{fig:Ex3mesh}
\end{figure}

\medskip \noindent\textbf{Example with jump of data on the interface.} 
First we verify the convergence rate of the proposed VE schemes with polynomial degree $k=2$. For this case we focus on the stationary case. We employ the method of manufactured solutions, for which we consider the following smooth closed-forms for global displacement and fluid pressure
\begin{align*}
\bu(x,y)=0.1\begin{pmatrix}
x(1-x) \cos(\pi x) \sin (\pi y)
\\
\sin(\pi y) \cos(\pi y) y^2(2-y)
\end{pmatrix}, \quad
p^P(x,y) =\sin( \pi x) \sin( \pi y),
\end{align*}
together with the parametric values (in adimensional form)
\begin{gather*}
\nu^{\rmP}=0.3, \quad E^{\rmP}=100, \quad \nu^{\rmE}=0.45, \quad E^{\rmE}=10000,\quad \kappa=10^{-6},\quad \alpha=0.1,\quad c_0=10^{-3},\quad \eta=0.01,
\end{gather*}
in the domain $\Omega_1:=(0,1)\times(0,2)$ 
and the Lam\'e constants are obtained
from the Young and Poisson moduli $E,\nu$ as $\lambda=\frac{ E  \nu}{(1+\nu)(1-2\nu)},\,
\mu= \frac{E}{2+2 \nu}$ on each subdomain $\OmP$ and $\OmE$. The manufactured displacement and pore pressure are used to construct manufactured global pressure (being defined separately on each subdomain). 
The poroelastic domain is $\OmP:=(0,1)^2$ and the elastic region is $\OmE:=(0,1)\times(1,2)$,
and we consider the same parametric properties in the whole domain $\Omega =\OmP \cup \OmE $,
and the exact global pressure is obtained from the respective problems in each subdomain. The boundary data are non-homogeneous, with values inherited from the manufactured displacement and pressures. The boundaries of the domain are set up as follows: $\Gamma_D^{\rmE}$ is the top segment, $\Gamma_D^{\rmP}$ is the bottom segment, $\Gamma_N^{\rmE}$ is conformed by the vertical segments of the elastic domain, and $\Gamma_N^{\rmP}$ corresponds to the vertical segments of the Biot boundary.  

Rectangular meshes (see Figure \ref{fig:sq_mesh})
are employed for each subdomain, but with different mesh sizes. We apply uniform mesh refinement on each discrete subdomain and generate successively refined meshes on which we compute approximate solutions, errors, and experimental convergence rates. 
We emphasise that for both the \emph{dofi-dofi} and edge stabilisations,
the discrete formulation achieves  optimal convergence rates as predicted by the theoretical error bounds \eqref{th4.3-est} (and considering only the steady case) as shown in Tables \ref{tab:test2} and \ref{tab:test2Stab}. We also show samples of the discrete displacements in Figure~\ref{fig:ex1-sols}. 
We also stress that the convergence is optimal even for the current parameter set which is in a challenging regime of near incompressibility and poroelastic locking. Other tests (not shown) confirm that this behaviour is also observed in a wide  range of material parameters. 
	
	\begin{figure}[t!]
			\begin{center}
				\subfigure[]{\includegraphics[width=0.495\textwidth]{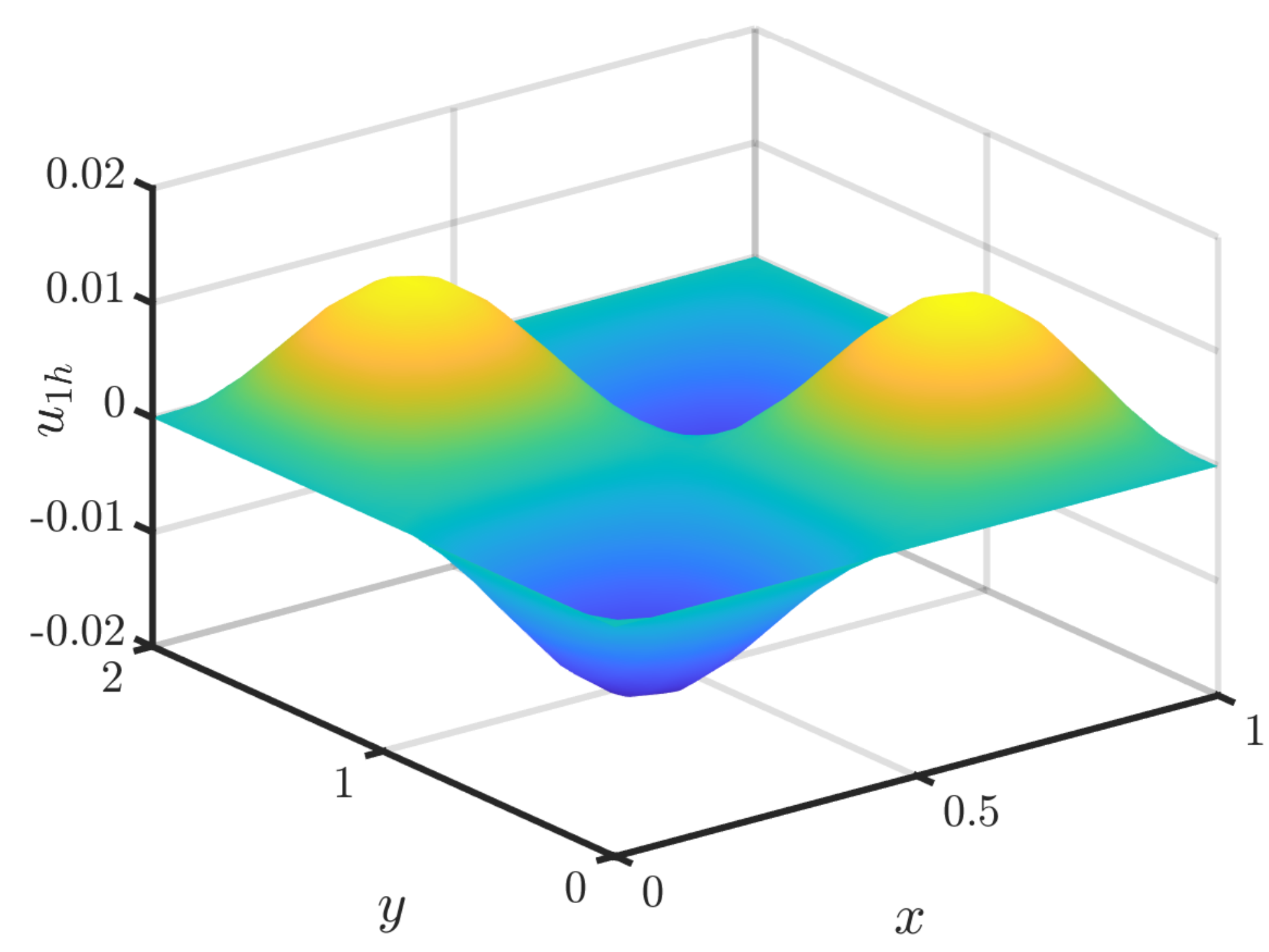}}
				\subfigure[]{\includegraphics[width=0.495\textwidth]{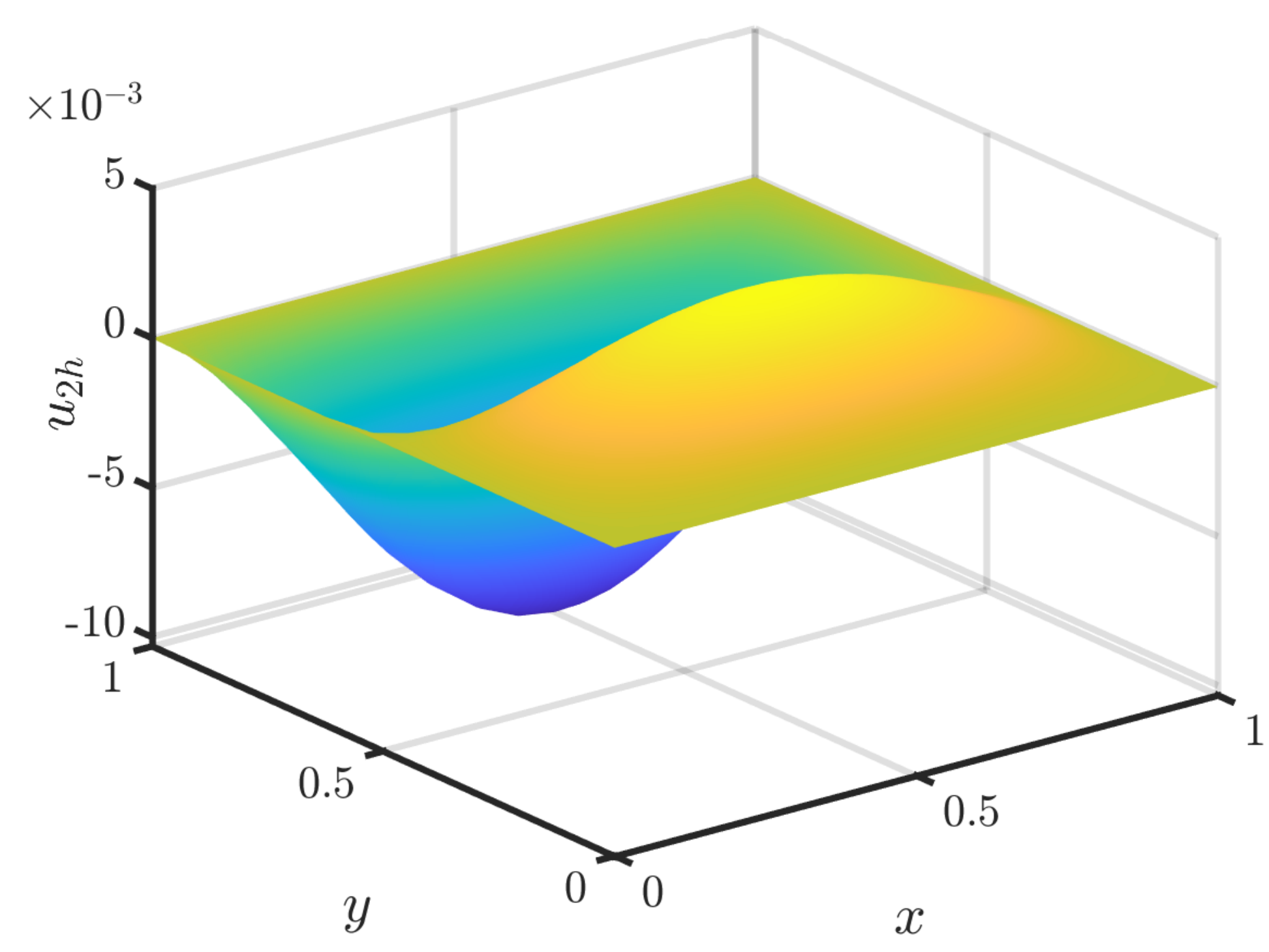}}
			\end{center}
			
			\vspace{-4mm}
			\caption{Accuracy verification test. Approximate displacement components.}
			\label{fig:ex1-sols}
		\end{figure}

\begin{table}[t!]
	\setlength{\tabcolsep}{5pt}
	\begin{center}
		\begin{tabular}{|ccccccccccc|}
			\hline
			$h$ & $E_0^2(\bu)$ & $r_0^2(\bu)$  & $E_1^2(\bu)$ & $r_1^2({\bu})$ & $E_0^2(p)$ & $r_0^2(p)$ & $E_1^2(p)$ & $r_1^2(p)$ & $E_0^2(\psi)$ & $r_0^2(\psi)$
			\\
			\hline
			0.47 & 0.07013 & -- & 0.21980 & -- & 0.27687 & -- & 0.46972 & -- & 0.15078 & -- \\
			0.28 & 0.01079 & 2.70 & 0.07060 & 1.64 & 0.05095 & 2.44 & 0.11743 & 2.00 & 0.05683 & 1.41
			\\
			0.16 & 1.57e-03 & 2.78 & 0.02042 & 1.79 & 4.13e-03 & 3.63 & 0.01940 & 2.60 & 	0.01793 & 1.66
			\\
			0.08 & 2.17e-04 & 2.85 & 5.53e-03 & 1.89 & 2.99e-04 & 3.79 & 4.08e-03 & 2.25 & 5.07e-03 & 1.82
			\\
			0.04 & 2.87e-05 & 2.91 & 1.44e-03 & 1.94 & 2.39e-05 & 3.64 & 9.68e-04 & 2.08 & 1.35e-03 & 1.91 \\
			0.02 & 3.70e-06 & 2.95 & 3.68e-04 & 1.97 & 2.35e-06 & 3.35 & 2.39e-04 & 2.02 & 3.47e-04 & 1.96 \\
			\hline
		\end{tabular}
	\end{center}
	\vspace{-3mm}
	\caption{Example with data jump on the interface. Verification of space convergence for the method with $k=2$ for the \emph{dofi-dofi} stabiliser. Errors and experimental convergence rates  for global displacement, global  pressure and Biot fluid pressure. 
	}\label{tab:test2}
\end{table}	
\begin{table}[t!]
	\setlength{\tabcolsep}{5pt}
	\begin{center}
		\begin{tabular}{|ccccccccccc|}
			\hline
			$h$ & $E_0^2(\bu)$ & $r_0^2(\bu)$  & $E_1^2(\bu)$ & $r_1^2({\bu})$ & $E_0^2(p)$ & $r_0^2(p)$ & $E_1^2(p)$ & $r_1^2(p)$ & $E_0^2(\psi)$ & $r_0^2(\psi)$
			\\
			\hline
			0.47 & 0.06444 & -- & 0.21567 & -- & 0.28577  & -- & 0.45811   & -- & 0.14884 & -- \\

			0.28 & 0.01050 & 2.62 & 0.07014 & 1.62 & 0.03749 & 2.93 & 0.10439 & 2.13 & 0.05667 & 1.39
			\\
			
			0.16 & 1.58e-03 & 2.73 & 0.02044 & 1.78 & 2.89e-03 & 3.70  & 0.01910 & 2.45 & 0.01792 & 1.66
			\\
			0.08 & 2.20e-04 & 2.84 & 5.56e-03 & 1.88 & 2.18e-04 & 3.72 & 4.08e-03 & 2.23 & 5.06e-03 & 1.82
			\\
			0.04 & 2.90e-05 & 2.92 & 1.45e-03 & 1.94 & 2.00e-05 & 3.45 & 9.68e-04 & 2.07 & 1.35e-03 & 1.91
			\\
			0.02 & 3.72e-06 & 2.96 & 3.70e-04  & 1.97 & 2.19e-06  &  3.18 & 2.39e-04  & 2.02 & 3.47e-04 & 1.96
			\\
			\hline
		\end{tabular}
	\end{center}
	\vspace{-3mm}
	\caption{Example with data jump on the interface. Verification of space convergence for the method with $k=2$ for the tangential edge stabiliser. Errors and convergence rates $r$ for global displacement, global  pressure and fluid pressure.}\label{tab:test2Stab}
\end{table}

	\begin{table}[t!]
		\setlength{\tabcolsep}{5pt}
		\begin{center}
			\begin{tabular}{|ccccccccccc|}
				\hline
				$h$ & $E_0^2(\bu)$ & $r_0^2(\bu)$  & $E_1^2(\bu)$ & $r_1^2({\bu})$ & $E_0^2(p)$ & $r_0^2(p)$ & $E_1^2(p)$ & $r_1^2(p)$ & $E_0^2(\psi)$ & $r_0^2(\psi)$				\\
				\hline
				0.282 & 0.02622 & -- & 0.11812 & -- & 0.00925 & -- & 0.05935 & -- &  0.14289 & --
				\\
				0.143 & 3.29e-03 & 2.99 & 0.03009 & 1.97 & 1.11e-03 & 3.06 & 0.01422 & 2.06 & 0.03672 & 1.96 \\
				0.072 & 4.13e-04 & 2.99 & 7.52e-03 & 1.99 & 1.37e-04 & 3.02 & 3.47e-03 & 2.03 & 9.26e-03 & 1.98 \\
				0.035 & 5.11e-05 & 3.00 & 1.87e-03 & 2.00 & 1.80e-05 & 2.96 & 8.71e-04 & 1.99 & 2.32e-03 & 1.99 \\
				0.017 & 6.42e-06 & 3.00 & 4.69e-04 & 2.00 & 2.14e-06 & 3.03 & 2.13e-04 & 2.03 & 5.81e-04 & 1.99 \\
				0.008 & 8.02e-07 & 3.00 & 1.17e-04 & 2.00 & 2.67e-07 & 3.01 & 5.31e-05 & 2.01 & 1.45e-04 & 1.99
				\\
				\hline
			\end{tabular}
		\end{center}
		\vspace{-3mm}
		\caption{Example with small edges. Verification of space convergence for the method with $k=2$ for the tangential edge 
		stabiliser. 
		Errors and convergence rates $r$ for global displacement, global  pressure, and fluid pressure.}\label{tab:test3}
\end{table}

\medskip\noindent \textbf{Example with small edges in the mesh.}
The aim of this test is to examine the influence of the mesh assumptions.We compare
the performance of the proposed VEM when the geometric assumption $\tilde{A}$ is violated (see Remark~\ref{hfgrtrdc}).
	
	For this test, we consider the following non-polynomial closed-forms for global displacement and Biot fluid pressure
	\begin{align*}
	\bu(x,y)=0.1\begin{pmatrix}
	x(1-x) \cos(\pi x) \sin (2 \pi y)
	\\
	\sin(\pi x) \cos(\pi y) y^2(1-y)
	\end{pmatrix}, \quad
	p^\mathrm{P}(x,y) =\cos(2 \pi (x-0.25)) \sin(2 \pi (y-0.25)),
	\end{align*}
	together with the non-dimensional parametric values
	\begin{gather*}
	\nu^{\rmP}=0.3, \quad E^{\rmP}=10, \quad \nu^{\rmE}=0.4, \quad E^{\rmE}=100,\quad \kappa=1,\quad \alpha=1,\quad c_0=1,\quad \eta=1,
	\end{gather*}
	in the domain $\Omega:=(0,1)^2$ and the Lam\'e constants are obtained
	from the Young and Poisson moduli $E,\nu$ as $\lambda=\frac{ E  \nu}{(1+\nu)(1-2\nu)},\,
	\mu= \frac{E}{2+2 \nu}$ on each independent domain $\OmP$ and $\OmE$.
	The elastic region is $\OmE:=(0.25,0.75)^2$ and the poroelastic
	domain is $\OmP:=\Omega \backslash \OmE$, and we consider the
	same parametric properties in the whole domain $\Omega $,
	and the exact global pressure is obtained from the respective problems
	in each subdomain. Polygonal meshes (see Figure~\ref{fig:Ex3mesh})
	are employed for each subdomain, but with different mesh sizes.
	
We emphasise that using the edge stabilisation yields the optimal convergence rates shown in Table~\ref{tab:test3} (obtained with the polynomial degree $2$).
	Even though the theoretical analysis has been developed under assumption $\tilde{A}$, this numerical
example shows that the error estimates might also hold true for more general mesh assumptions.
	
\medskip\noindent\textbf{Testing the convergence of a space and time dependent solution.}
Next we concentrate on a time-dependent problem featuring manufactured solutions for displacement and fluid pressure given by $\bu=\sin(t)(x^2 +y^2)[1,1]$ and $p=\sin(t)(x^2 +y^2)$, respectively. By considering the governing equations, we obtain the corresponding load functions using the provided physical parameters: $\nu^{\rmP}=0.3$, $E^{\rmP}=1$, $\nu^{\rmE}=0.4$, $E^{\rmE}=1$, $\kappa=1$, $\alpha=1$, $c_0=1$, and $\eta=1$. 

To assess the convergence properties of the numerical method employed, we systematically vary the mesh size $h$ on polygonal meshes (see Figure~\ref{fig:Ex3mesh}) and take the time step $\Delta t = h^2$, and analyse their impact on the solution accuracy. Table~\ref{tab:test4} confirms that the VEM analysed in previous sections  
 achieve the optimal second-order rate of convergence with respect to $h$ and $\Delta t$ predicted by Theorem~\ref{thm:fullyd}. 
	
	\begin{table}[t!]
	\setlength{\tabcolsep}{5pt}
	\begin{center}
		\begin{tabular}{|cccccccc|}
			\hline
			$h$ & $\Delta t$ & $E_1^2(\bu)$ & $r_1^2({\bu})$  & $E_1^2(p)$ & $r_1^2(p)$ & $E_0^2(\psi)$ & $r_0^2(\psi)$				\\
			\hline
		0.2795 & 0.07813 & 5.10e-03 & --  & 2.41e-03 & -- & 4.75e-03 & -- \\ 
		0.1398 & 0.01953 & 1.35e-03 & 1.92 & 6.17e-04 & 1.97 & 1.25e-03 & 1.92 \\ 
		0.0698 & 4.88e-03 & 3.36e-04 & 2.01 & 1.55e-04 & 1.98 & 3.12e-04 & 2.01 \\ 
		0.0349 & 1.22e-03 & 8.42e-05 & 1.99 & 3.87e-05 & 1.99 & 7.82e-05 & 1.99 \\ 
		0.0174 & 3.51e-04 & 2.11e-05 & 2.00 & 9.68e-06 & 2.00 & 1.95e-05  & 2.00 \\		
			\hline
		\end{tabular}
	\end{center}
	\vspace{-3mm}
	\caption{Verification of convergence rates $r$ for global displacement, global  pressure, and fluid pressure with varying mesh size $h$ and time step $\Delta t$.}\label{tab:test4}
\end{table}

\medskip\noindent\textbf{Example with a non-straight interface.} 
		In this test we investigate again the convergence of the method but now considering an interface as shown in Figure \ref{fig:circ_mesh}(a), approximating the circle of radius $\frac{1}{4}$ centred at $(\frac12,\frac12)$  separating elastic and poroelastic subdomains in the domain $\Omega=(0,1)^2$. The exact solutions for displacement and fluid pressure are chosen  as in the previous example, now selecting the following non-dimensional parametric data 
		\begin{gather*}
		\nu^{\rmP}=0.49999, \quad E^{\rmP}=100, \quad \nu^{\rmE}=0.499, \quad E^{\rmE}=\text{3e+04},\quad \kappa=\text{1e-04},\quad \alpha=1,\quad c_0=\text{1e-03},\quad \eta=1.
		\end{gather*}
		
\begin{figure}[t!]
			\begin{center}
				\subfigure[]{\includegraphics[width=0.4\textwidth]{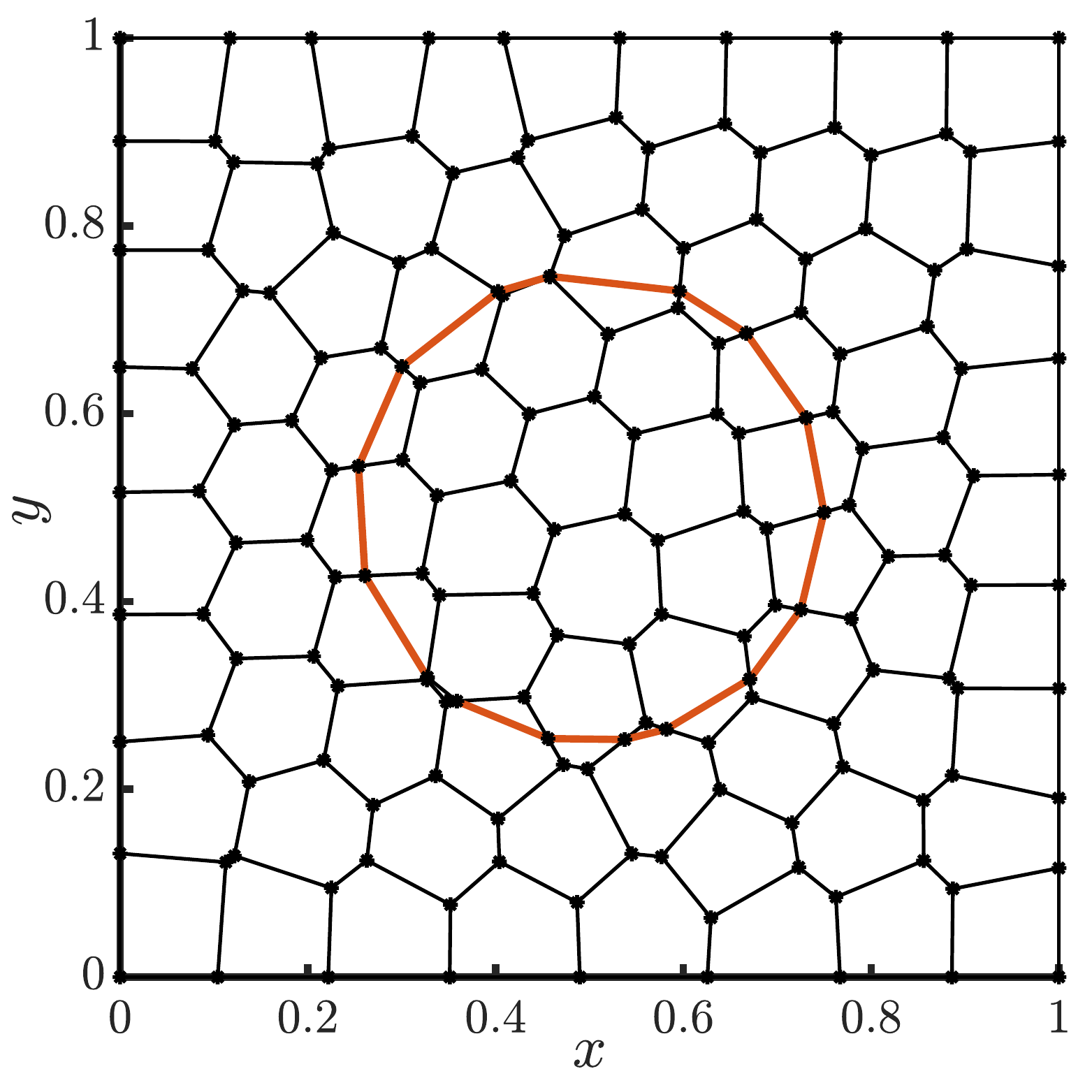}}\\
				\subfigure[]{\includegraphics[width=0.325\textwidth]{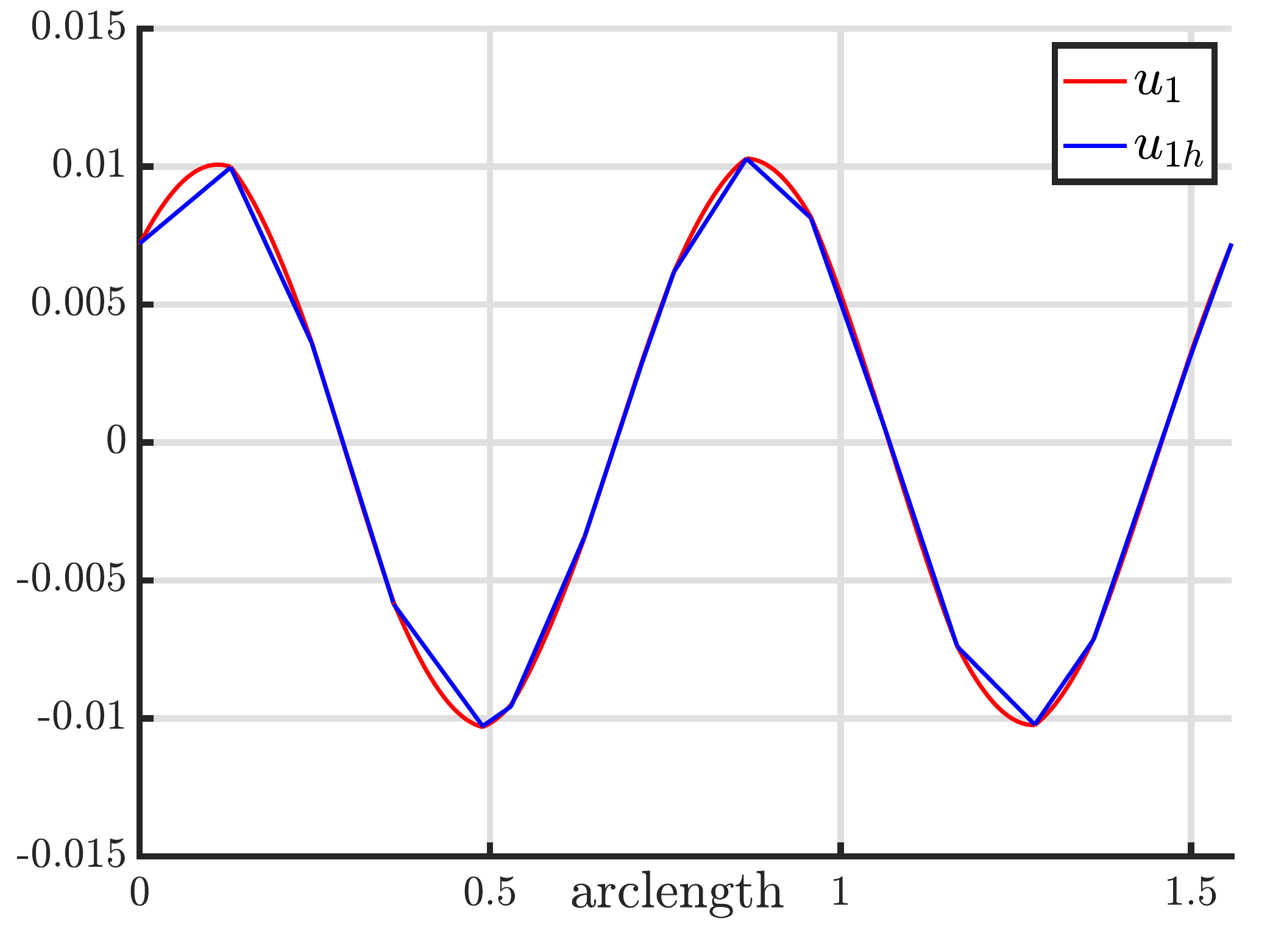}}
			\subfigure[]{\includegraphics[width=0.325\textwidth]{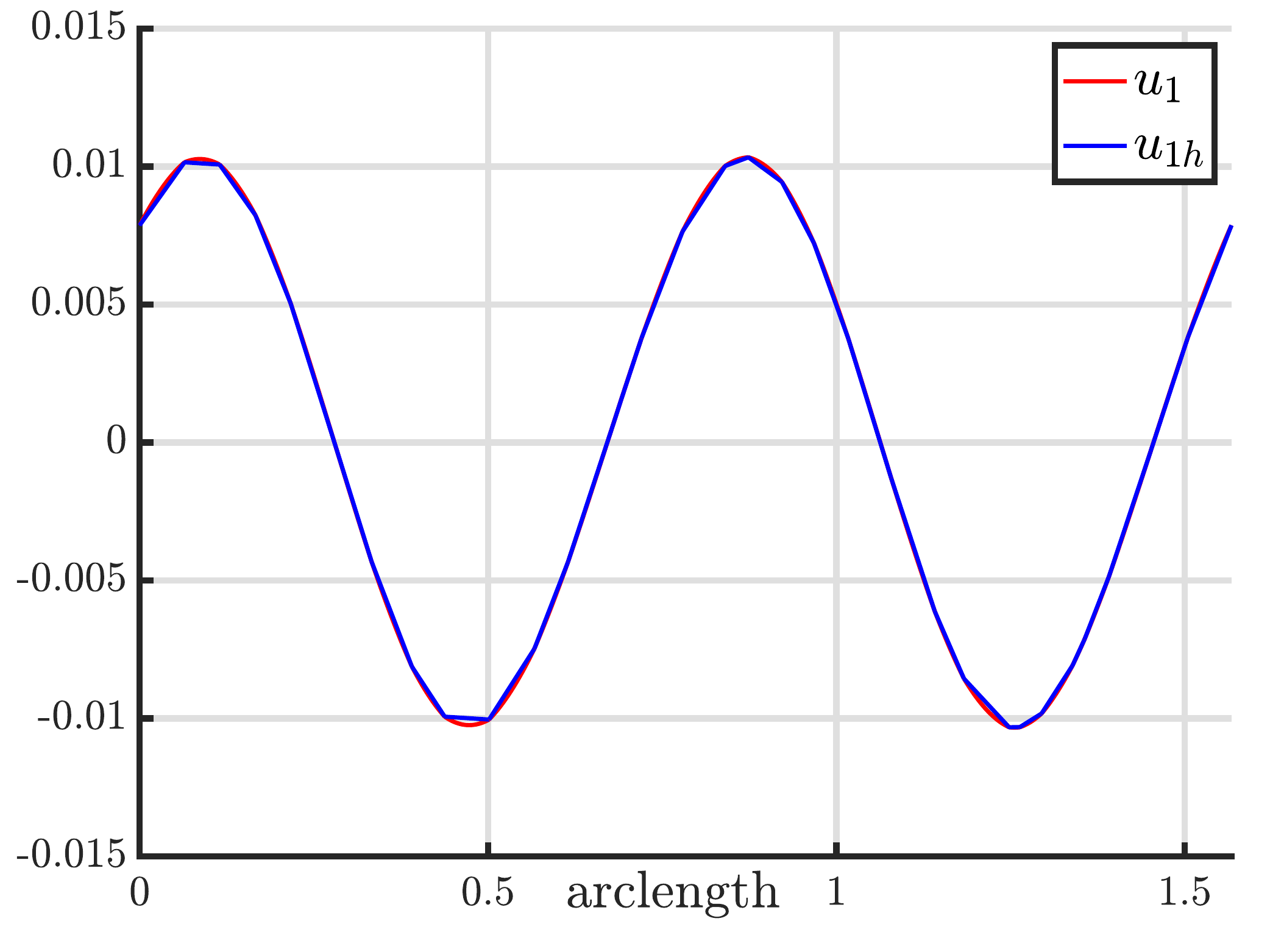}}
			\subfigure[]{\includegraphics[width=0.325\textwidth]{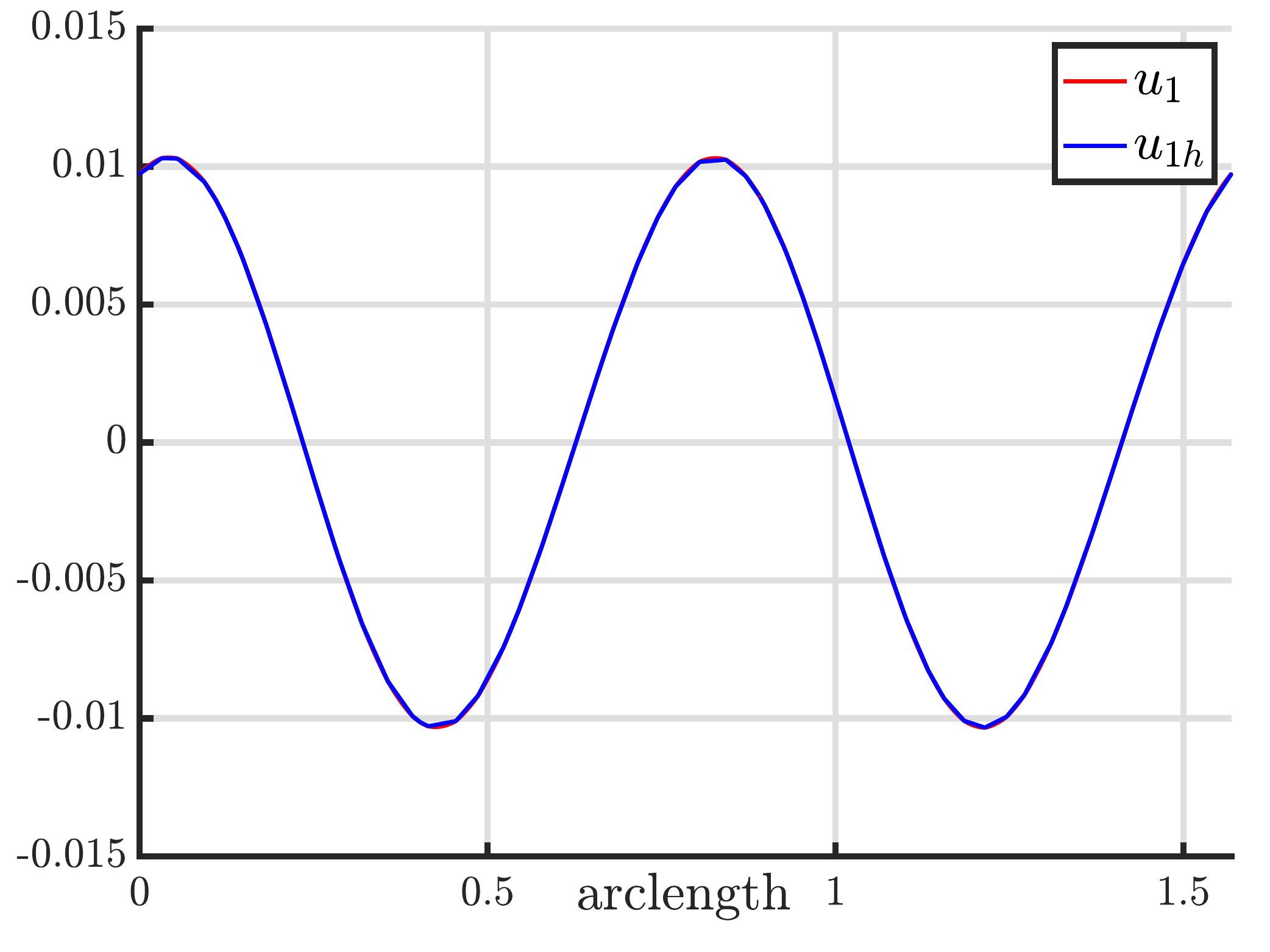}}
			\end{center}
			
			\vspace{-4mm}
			\caption{(a) Polygonal mesh with smaller edges at the interface on the domain $\Omega$, and first component of exact displacement solution (red) and approximate solution (blue) at polygonal interface with mesh size (b) $h=0.14$ (c) $h=0.07$ (d) $h=0.035$.}						\label{fig:circ_mesh}
		\end{figure}

Panels (b)-(c)-(d) of Figure \ref{fig:circ_mesh} show the approximate discrete first component of displacement and the exact horizontal displacement along the arc-length of the interface, for three different mesh resolutions. This confirms convergence with mesh refinement.

\medskip\noindent\textbf{Mandel test.} 
To conclude this section we consider here a modification of the classical Mandel's problem to the case of an interface elastic-poroelastic material \cite{adgmr20}. The domain is  the rectangle $\Omega = (0,100)\times(0,40)$ (in m$^2$) and we use the following parametric values
\begin{gather*}
\nu^{\rmP}=0.4, \quad E^{\rmP}=\text{2.4e+05}\,\mathrm{Pa}, \quad \nu^{\rmE}=0.499, \quad E^{\rmE}=\text{4.8e+05}\,\mathrm{Pa},\\ \kappa=\text{1e-06}\,\mathrm{m}^2,\quad \alpha=1,\quad c_0=\text{2.5e-04}\,\mathrm{Pa}^{-1},\quad \eta=\text{1e-03}\,\mathrm{m}^2/\mathrm{s}.
\end{gather*}
In this test, a constant traction of magnitude 5e+4 is applied on the top boundary of domain to produce the compression shown in Figure \ref{fig:mandel-displacement}(a)-(b). Sliding conditions are imposed along the left and bottom boundaries of the domain. The boundary conditions for the pore pressure field are taken as homogeneous Dirichlet on the right boundary of the poroelastic region, and of zero-flux type on the rest of the poroelastic boundary. This problem does not have a closed-form solution, but a qualitative agreement with the results from \cite{adgmr20} (in terms of deformation patterns) is observed in the figure. We also note there the displacement is stable even for the critical case of small specific storage coefficient $c_0$ and large $\lambda$.

\begin{figure}[t!]
			\begin{center}
				\subfigure[]{\includegraphics[width=0.495\textwidth]{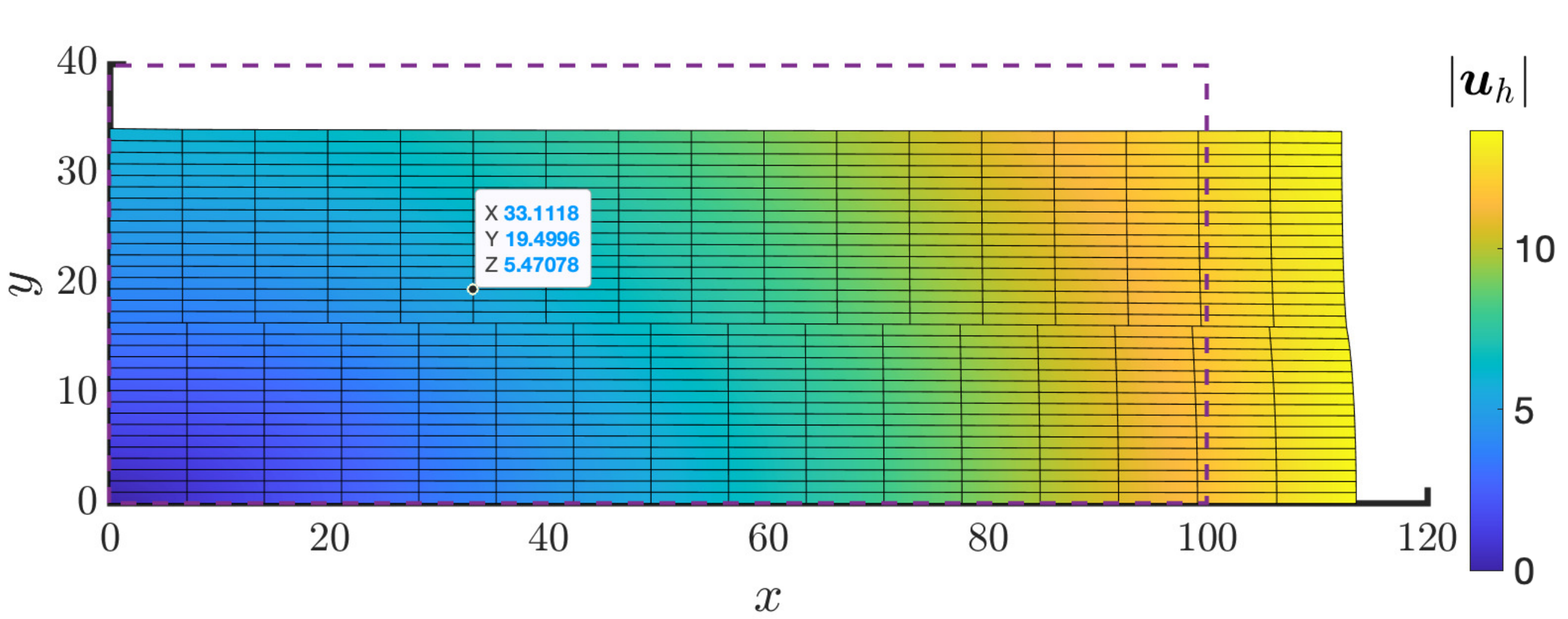}}
				\subfigure[]{\includegraphics[width=0.495\textwidth]{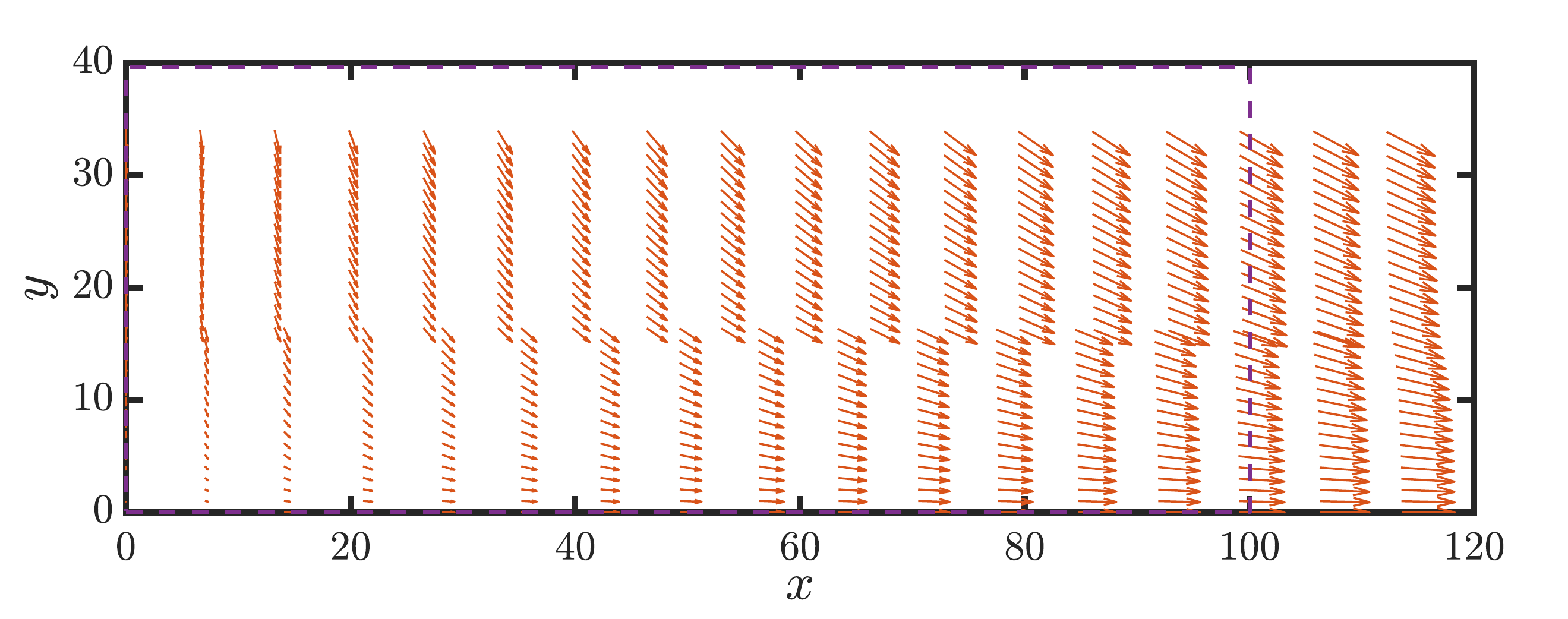}}
			\end{center}
			
			\vspace{-4mm}
			\caption{Magnitude of approximate displacement (a) and displacement vectors (b) rendered on the deformed configuration. For reference, the boundary of the undeformed domain is represented in dashed lines in both panels. }
			\label{fig:mandel-displacement}
		\end{figure}


\begin{appendices}

\section{Proof of Theorem~\ref{th:semid}} \label{proof-semi}
Thanks to the Scott--Dupont theory (see \cite{brenner08})
we know that
for every $s$ with $0\le s\le k$ and for every $u\in H^{1+s}(K)$,
there exists $u_\pi\in\mathbb{P}_k(K)$, $k\ge2$, such that
\begin{equation}\label{poly_est_u}
\| u - u_{\pi} \|_{0,K} + h_K |u -u_{\pi} |_{1,K} \lesssim h_K^{1+s} |u|_{1+s,K}
\quad \text{for all $K \in \mathcal{T}_h$.}
\end{equation}
We can then write the displacement and total pressure error
in terms of the poro\-elastic projector
\begin{align*}
(\bu - \bu_h)(t) = (\bu - I^h_{\bu}\bu)(t) + (I^h_{\bu} \bu - \bu_h)(t) := e_{\bu}^I(t) + e_{\bu}^A(t), \\
(\psi - \psi_h)(t) = (\psi - I^h_{\psi} \psi)(t) + (I^h_{\psi} \psi - \psi_h)(t) := e_{\psi}^I(t) + e_{\psi}^A(t).
\end{align*}
Then, a combination of equations \eqref{weak-Iuh}, \eqref{weak-uh} and \eqref{weak-u} gives
\begin{align*}
a_1^h(e_{\bu}^A, \bv_h)+ b_1(\bv_h, e_{\psi}^A) & = (a_1(\bu,\bv_h) -a_1^h(\bu_h, \bv_h)) + b_1(\bv_h, \psi- \psi_h) 
= (F-F^h)(\bv_h),
\end{align*}
and taking as test function $\bv_h = \partial_t e_{\bu}^A$, we
can write the relation
\begin{align}
a_1^h(e_{\bu}^A, \partial_t e_{\bu}^A)
+ b_1(\partial_t e_{\bu}^A, e_{\psi}^A) = (F-F^h)(\partial_t e_{\bu}^A). \label{est-eq1}
\end{align}
Now, we write the pressure error in terms of the poroelastic projector as follows
\begin{align*}
(p^{\rmP} - p_h^{\rmP})(t) = (p^{\rmP} - I^h_p p^{\rmP})(t) + (I^h_p p^{\rmP} - p_h^{\rmP})(t) := e_{p}^{I,\rmP}(t) + e_{p}^{A, \rmP}(t).
\end{align*}
Using \eqref{weak-Iph}, \eqref{weak-ph} and \eqref{weak-p}, we obtain
\begin{align*}
& \tilde{a}_2^h(\partial_t e_p^{A,\rmP}, q_h^{\rmP}) + a_2^h(e_p^{A, \rmP}, q_h^{\rmP}) - b_2(q_h^{\rmP}, \partial_t e_{\psi}^{A} ) \\ 
& =  (\tilde{a}_2^h(\partial_t I^h_p p^{\rmP}, q_h^{\rmP}) - \tilde{a}_2(\partial_t p^{\rmP}, q_h^{\rmP})) + b_2(q_h^{\rmP}, \partial_t e_{\psi}^I)  + (G- G^h)(q_h^{\rmP}).
\end{align*}
We can take $q_h^{\rmP} = e_p^{A, \rmP}$, which leads to
\begin{align} \label{est-eq2} \begin{split}
& \tilde{a}_2^h(\partial_t e_p^{A, \rmP}, e_p^{A, \rmP}) +a_2^h(e_p^{A, \rmP}, e_p^{A, \rmP}) - b_2(e_p^{A, \rmP}, \partial_t e_{\psi}^A) \\ & = (\tilde{a}_2^h(\partial_t I^h_p p^{\rmP}, e_p^{A, \rmP}) - \tilde{a}_2(\partial_t p^{\rmP}, e_p^{A, \rmP}) ) + b_2(e_p^{A, \rmP}, \partial_t e_{\psi}^I) + (G- G^h)(e_p^{A, \rmP}). \end{split}
\end{align}
Next we use \eqref{weak-Ipsih}, \eqref{weak-psih} and \eqref{weak-psi}, and this implies
\begin{align*}
& b_1(e_{\bu}^A, \phi_h)+ b_2(e_p^{A, \rmP} , \phi_h) - a_3(e_{\psi}^A, \phi_h)
= -b_2(e_p^{I, \rmP}, \phi_h) + a_3(e_{\psi}^I, \phi_h).
\end{align*}
Differentiating the above equation with respect to time
and taking $\phi_h = -e_{\psi}^A$, we can assert that
\begin{align}
- b_1(\partial_t e_{\bu}^A, e_{\psi}^A) - b_2(\partial_t e_p^{A, \rmP}, e_{\psi}^A) + a_3(\partial_t e_{\psi}^A, e_{\psi}^A) = b_2(\partial_t e_p^I, e_{\psi}^A) -a_3(\partial_t e_{\psi}^I, e_{\psi}^A). \label{est-eq3}
\end{align}
Then we simply add \eqref{est-eq1}, \eqref{est-eq2} and \eqref{est-eq3}, to obtain
\begin{align} \label{est-eq4} \begin{split}
&a_1^h(e_{\bu}^A, \partial_t e_{\bu}^A)  + \tilde{a}_2^h(\partial_t e_p^{A, \rmP}, e_p^{A, \rmP}) +a_2^h(e_p^{A, \rmP}, e_p^{A, \rmP}) + a_3(\partial_t e_{\psi}^A, e_{\psi}^A) - b_2(e_p^{A, \rmP}, \partial_t e_{\psi}^A) - b_2(\partial_t e_p^{A, \rmP}, e_{\psi}^A)\\
& = (F- F^h)(\partial_{t} e_{\bu}^A ) + (\tilde{a}_2^h(\partial_t I^h_p p^{ \rmP}, e_p^{A, \rmP}) - \tilde{a}_2(\partial_{t} p^{\rmP}, e_p^{A, \rmP}))
\\& \quad + b_2(e_p^{A, \rmP}, \partial_{t} e_{\psi}^I) + (G- G^h) (e_p^{A, \rmP})  + b_2(\partial_{t} e_p^{I, \rmP}, e_{\psi}^A) - a_3(\partial_{t} e_{\psi}^I, e_{\psi}^A). \end{split}
\end{align}
Regarding the left-hand side of \eqref{est-eq4}, repeating arguments to obtain alike to
	the stability proof
That is,
\begin{align*}
& a_1^h(e_{\bu}^A, \partial_t e_{\bu}^A)  + \tilde{a}_2^h(\partial_t e_p^{A, \rmP}, e_p^{A, \rmP})
+a_2^h(e_p^{A, \rmP}, e_p^{A, \rmP}) + a_3(\partial_t e_{\psi}^A, e_{\psi}^A) - b_2(e_p^{A, \rmP}, \partial_t e_{\psi}^A) - b_2(\partial_t e_p^{A, \rmP}, e_{\psi}^A)
\\
&  \gtrsim   {\mu^{\min}}\frac{\mathrm{d}}{\mathrm{d}t} \| \beps(e_{\bu}^A)\|_0^2 + c_0\frac{\mathrm{d}}{\mathrm{d}t} \| e_p^{A, \rmP}\|_{0, \OmP}^2 + \frac{2 \kappa_{\min}}{\eta} \|\nabla e_p^A\|_{0, \OmP}^2 + \frac{1}{{ \lambda^{\rmE}}} \frac{\mathrm{d}}{\mathrm{d}t} \|e_{\psi}^{A, \rmE}\|_{0, \OmE}^2 \\
& \qquad \qquad
+ \frac{1}{{ \lambda^{\rmP}}} \sum_{K \in \cT_h^{\rmP}} \biggl( \alpha^2\frac{\mathrm{d}}{\mathrm{d}t}\| (I-\Pi^0_K)  e_p^{A,\rmP}\|_{0,K}^2
+\frac{\mathrm{d}}{\mathrm{d}t} \|\alpha \Pi^0_K e_p^{A,\rmP} -e_{\psi}^{A,\rmP}\|_{0,K}^2  \biggr) .
\end{align*}
Then integrating equation \eqref{est-eq4} in time and using the consistency of   $\tilde{a}_2(\cdot, \cdot)$, implies the bound
\begin{align*}
& {\mu_{\min}} \| \beps(e_{\bu}^A(t))\|_0^2 + c_0  \| e_p^{A, \rmP} (t)\|_{0, \OmP}^2 + \frac{1}{{ \lambda^{\rmE}}} \|e_{\psi}^{A, \rmE}(t)\|_{0, \OmE}^2  + \frac{\kappa_{\min}}{\eta}\int_0^t \|\nabla e_p^{A, \rmP}(s)\|_{0, \OmP}^2 \, \mathrm{d}s \\
& + \frac{1}{{ \lambda^{\rmP}}} \sum_{K \in \cT_h^{\rmP}} \Bigl(\alpha^2  \| (I-\Pi^{0,k}_K)  e_p^{A, \rmP}(t)\|_{0,K}^2
+  \|(\alpha \Pi^{0,k}_K e_p^{A, \rmP} -e_{\psi}^{A, \rmP})(t)\|_{0,K}^2  \Bigr) \\
& \lesssim {\mu_{\min}} \| \beps(e_{\bu}^A(0))\|_0^2 + c_0  \| e_p^{A, \rmP} (0)\|_{0, \OmP}^2 + \frac{1}{{ \lambda^{\rmE}}} \|e_{\psi}^{A, \rmE}(0)\|_{0, \OmE}^2
\\& \quad
+ \frac{1}{{ \lambda^{\rmP}}} \sum_{K \in \cT_h^{\rmP}} \Bigl(\alpha^2  \| (I-\Pi^{0,k}_K)  e_p^{A, \rmP}(0)\|_{0,K}^2 +  \|(\alpha \Pi^{0,k}_K e_p^{A, \rmP} -e_{\psi}^A)(0)\|_{0,K}^2  \Bigr) \\
& \quad + \underbrace{ \int_0^t  \bigl((\bb- \bb_h)(s),\partial_{t} e_{\bu}^A(s) \bigr)_{0, \Omega} \, \mathrm{d}s}_{=:D_1} + \underbrace{\int_0^t \bigl((\ell^{\rmP} - \ell_h^{ \rmP})(s), e_p^{A, \rmP}(s) \bigr)_{0, \OmP}\,  \mathrm{d}s}_{=:D_2} \\
& \quad + \underbrace{\int_0^t \sum_{K \in \cT_h^{\rmP}} \Bigl(\tilde{a}_2^{h,K} \bigl(\partial_t(I^h_p p^{\rmP} - p_{\pi}^{\rmP})(s), e_p^{A, \rmP}(s) \bigr) - \tilde{a}^{K}_2 \bigl(\partial_t(p^{\rmP} - p_{\pi}^{\rmP})(s), e_p^{A, \rmP}(s) \bigr) \Bigr) \, \mathrm{d}s}_{=:D_3} \\
& \quad + \underbrace{\int_0^t \Bigl(b_2 \bigl(e_p^{A, \rmP}(s), \partial_{t} e_{\psi}^I(s) \bigr) + b_2 \bigl(\partial_{t} e_p^{I, \rmP} (s), e_{\psi}^A (s)\bigr) - a_3 \bigl(\partial_{t} e_{\psi}^I(s), e_{\psi}^A(s) \bigr) \Bigr) \, \mathrm{d}s}_{=:D_4}.
\end{align*}
Then we  integrate by parts  in time, and use Cauchy--Schwarz and Young's inequalities to arrive at
\begin{align*}
D_1 
& \le  \frac{{\mu_{\min}} }{2} \| \beps(e_{\bu}^A(t))\|_0^2 + C_1(\mu)  h^k \biggl(h^k |\bb(t)|_{k-1}^2 + |\bb(0)|_{k-1} \| \beps(e_{\bu}^A(0))\|_0 + \int_0^t | \partial_{t} \bb (s)|_{k-1} \| \beps(e_{\bu}^A(s))\|_{0}\, \mathrm{d}s \biggr),
\end{align*}
where we have used standard error estimate for the
$L^2$-projection $\bPi_K^{0,k}$ onto piecewise constant functions.
Using again Cauchy--Schwarz inequality, standard  error
estimates for $\Pi_K^{0,k}$ on the term $D_2$, Young's
and Poincar\'e inequalities readily gives
\begin{align*}
D_2 \lesssim h^k \int_0^t |\ell^{\rmP} (s)|_{k-1, \OmP} \|\nabla e_p^{A, \rmP}(s)\|_{0, \OmP}\, \ds	\le C_2 h^{2k} \int_{0}^t |\ell^{\rmP} (s)|_{k-1, \OmP}^2 \ds + \frac{\kappa_{\min}}{6 \eta} \int_{0}^t \|\nabla e_p^{A, \rmP}(s)\|_{0, \OmP}^2 \ds.
\end{align*}
On the other hand, considering the polynomial approximation
$p_{\pi}^{\rmP}$ (cf. \eqref{poly_est_u}) of $p^{\rmP}$, utilising the triangle inequality, Young's
and Poincar\'e inequality yield
\begin{align*}
D_3 
&  \lesssim h^{2(k+1)} \bigg(c_0 + \frac{\alpha^2}{{\lambda^{\rmP}}}\bigg)^2 \int_0^t |\partial_{t} p^{\rmP}(s)|_{k+1, \OmP}^2 \ds + \frac{\kappa_{\min}}{ 6\eta} \int_{0}^t \|\nabla e_p^{A, \rmP}(s)\|_{0, \OmP}^2 \ds.
\end{align*}
Also,
\begin{align*}
D_4 
& \lesssim  \frac{1}{\lambda}   h^k \int_0^t  \Bigl( \alpha (
| \partial_{t} \psi^{\rmP} (s)|_{k, \OmP} +
|\partial_t \bu^{\rmP}(s)|_{k+1,\OmP} + |\partial_t \bu^{\rmE}(s)|_{k+1, \OmE} )\| e_p^{A, \rmP} (s)\|_{0, \OmP} \\ & \qquad \qquad  \qquad + (\alpha h |\partial_{t} p^{\rmP}(s) |_{k+1, \OmP} + | \partial_{t} \psi^{\rmP} (s)|_{k, \OmP} + | \partial_{t} \psi^{\rmE} (s)|_{k, \OmE}
+ |\partial_t \bu(s)|_{k+1}) \|e_{\psi}^A (s)\|_0 \Bigr) \, \mathrm{d}s.
\end{align*}
Using \eqref{discr-infsup} and a combination of equations \eqref{weak-Iuh},
\eqref{weak-uh} and \eqref{weak-u}, we get
\begin{align} \label{bound:semi_inf-sup}
\| e_{\psi}^A(t) \|_0 & \le \sup_{\bv_h \in \bV_h} \frac{b_1(\bv_h,  e_{\psi}^A(t))}{\| \bv_h \|_1}
\lesssim  h^k | \bb(t) |_{k-1} + {\mu_{\max}} \| \beps(e_{\bu}^A(t))\|_0 .
\end{align}
Then, with the help of Young's and Poincar\'e
inequalities, the bound of~$D_4$  becomes
\begin{align*}
D_4 \lesssim  \frac{1}{\lambda} h^k \int_0^t \Bigl( &(\alpha h^k |\partial_{t} p^{\rmP}(s) |_{k+1, \OmP}
+ | \partial_{t} \psi^{\rmP} (s)|_{k, \OmP} + | \partial_{t} \psi^{\rmE} (s)|_{k, \OmE} + |\partial_t \bu^{\rmP}(s)|_{k+1, \OmP} + |\partial_t \bu^{\rmE}(s)|_{k+1, \OmE}
)  \\
& \times ( h^k |\bb(s)|_k + {\mu_{\max}} \| \beps(e_{\bu}^A(s))\|_0)
 + \alpha \| e_p^{A, \rmP} (s)\|_{0, \OmP} (| \partial_{t} \psi^{\rmP} (s)|_{k, \OmP} +
|\partial_t \bu(s)|_{k+1} ) \Bigr) \ds.
\end{align*}
Combining the bounds of all $D_i, i=1,2,3,4$ and proceeding in a similar fashion as for the bounds in the stability proof in \cite{burger_acom21} 
(using Lemma \ref{lem:semi-Xbound} and \eqref{new:semi-Xbound}), we can eventually conclude that
\begin{align*}
&  {\mu_{\min}} \| \beps(e_{\bu}^A(t))\|_0^2 + c_0 \| e_p^{A,\rmP} (t)\|_{0, \OmP}^2 +  \frac{1}{{\lambda^{\rmE}}} \|e_{\psi}^{A, \rmE}(t)\|_{0, \OmE}^2 + \frac{\kappa_{\min}}{\eta} \int_0^t \| \nabla e_p^{A, \rmP} (s)\|_{0, \OmP}^2  \, \mathrm{d} s
\\
& \lesssim {\mu_{\min}} \| \beps(e_{\bu}^A(0))\|_0^2 + \Big(c_0 + \frac{\alpha^2}{{\lambda^{\rmP}}} \Big)  \| e_p^{A, \rmP} (0)\|_{0, \OmP}^2 +  \frac{1}{{\lambda^{\rmE}}} \|e_{\psi}^{A,\rmE}(0)\|_{0, \OmE}^2  \\
& \quad +  h^{k+1} \bigg( \sup_{t \in [0,t_{\text{final}}] } |\bb(t)|_{k-1}^2 + \int_0^t \biggl( |\bb(s)|_{k-1}^2 + | \partial_{t}\bb (s)|_{k-1}^2 + |\ell^{\rmP} (s)|_{k-1, \OmP}^2 \\
& \qquad \qquad \quad +  \frac{1}{{\lambda_{\min}^2}} \big( | \partial_{t} \psi^{\rmP} (s)|_{k, \OmP}^2 +| \partial_{t} \psi^{\rmE} (s)|_{k, \OmE}^2 +|\partial_t \bu^{\rmP}(s)|_{k+1, \OmP}^2 +|\partial_t \bu^{\rmE}(s)|_{k+1, \OmE}^2  \big) \\
& \qquad \qquad \quad
 + \Big( c_0 + \frac{\alpha^2}{{\lambda^{\rmP}}}\Big)^2 h^{k+1} |\partial_{t} p^{\rmP}(s) |_{k+1, \OmP}^2 \biggr) \, \mathrm{d}s \biggr).
\end{align*}
Then choosing $\bu_h(0): =\bu_I(0)$,
$\psi_h(0): = \Pi^{0,k-1}\psi(0)$, $p_h^{\rmP}(0): = p_I^{\rmP}(0)$ and applying
the triangle inequality together with \eqref{bound:semi_inf-sup}, completes the rest of the proof.

\section{Proof of Theorem~\ref{thm:fullyd}}  \label{proof-full}
As in the semidiscrete case we split the individual errors as
\begin{align*}
\bu(t_n) - \bu_h^n & = (\bu(t_n) - I^h_{\bu} \bu(t_n)) + (I^h_{\bu} \bu(t_n)- \bu_h^n)=: E_{\bu}^{I,n} + E_{\bu}^{A,n}, \\
\psi(t_n) - \psi_h^n &= (\psi(t_n) - I^h_{\psi} \psi(t_n)) + (I^h_{\psi} \psi(t_n)- \psi_h^n)=: E_{\psi}^{I,n} + E_{\psi}^{A,n}, \\
p^{\rmP}(t_n) - p_h^{n,\rmP} &= (p^{\rmP}(t_n) - I^h_{p} p^{\rmP}(t_n)) + (I^h_{p} p^{\rmP}(t_n)- p_h^{n,\rmP})=: E_{p}^{I,n} + E_{p}^{A,n},
\end{align*}
where the error terms are $E_{p}^{I,n}:=E_{p}^{I,n} |_{\OmP}, E_{p}^{A,n} := E_{p}^{A,n} |_{\OmP}$. Then, from estimate \eqref{estimate-Ihu-psi} and following the steps of the proof of
Theorem~\ref{th:semid} we get the bounds
\begin{subequations}
	\begin{align}
	\|E_{\bu}^{I,n} \|_1 
	&\lesssim  h^k ( | \bu(0) |_{k+1} + | \psi^{\rmP}(0) |_{k,\OmP} + | \psi^{\rmE}(0) |_{k,\OmE}  + \|\partial_t\bu\|_{\bL^1(0,t_n; [H^{k+1}(\Omega)]^2)}
	+ \|\partial_t \psi\|_{L^1(0,t_n; k)} ),\label{estimate-E_u^I} \\
	\label{estimate-E_{psi}^I}
	\|E_{\psi}^{I,n} \|_0 & \lesssim  h^k ( | \bu(0) |_{k+1} + | \psi^{\rmP}(0) |_{k,\OmP} + | \psi^{\rmE}(0) |_{k,\OmE}  + \|\partial_t\bu\|_{\bL^1(0,t_n; [H^{k+1}(\Omega)]^2)} + \|\partial_t\psi\|_{L^1(0,t_n; k)} ),\\
	\label{estimate-E_p^I}
	\|E_{p}^{I,n} \|_{1, \OmP} &\lesssim  h^k ( | p^{\rmP}(0) |_{k+1, \OmP} + \|\partial_t p^{\rmP}\|_{L^1(0,t_n; H^{k+1}(\OmP))}),
	\end{align}
\end{subequations}
where $\|\partial_t \psi\|_{L^1(0,t_n;k)}:= \|\partial_t \psi^{\rmP}\|_{L^1(0,t_n; H^k(\OmP))} + \|\partial_t \psi^{\rmE}\|_{L^1(0,t_n; H^k(\OmE))}$.
From equations \eqref{weak-Iuh}, \eqref{weak-uhn} and \eqref{weak-u}, we readily get
\begin{align} \label{estimate-eq1}
a_1^h(E_{\bu}^{A,n},\bv_h) + b_1(\bv_h, E_{\psi}^{A,n}) = F^n(\bv_h) - F^{h,n}(\bv_h).
\end{align}
We then use \eqref{weak-Ipsih} and \eqref{weak-psihn}, and proceed to differentiate \eqref{weak-psi} with respect to time. This implies
\begin{align} \label{estimate-eq2} \begin{split}
&b_1(E_{\bu}^{A,n} - E_{\bu}^{A,n-1}, \phi_h) + b_2(E_p^{A,n}-E_p^{A,n-1}, \phi_h) - a_3(E_{\psi}^{A,n} - E_{\psi}^{A,n-1}, \phi_h)  \\
& \quad = b_1((\bu(t_n) - \bu(t_{n-1})) - (\Delta t) \partial_{t} \bu(t_n), \phi_h)
+ b_2((I_p^h p^{\rmP}(t_n) - I_p^h p^{\rmP}(t_{n-1})) - (\Delta t) \partial_{t} p^{\rmP}(t_n), \phi_h) \\
& \qquad
- a_3((I_{\psi}^h \psi(t_n) - I_{\psi}^h \psi(t_{n-1})) - (\Delta t) \partial_{t} \psi (t_n), \phi_h).
\end{split} \end{align}
Choosing $\bv_h = E_{\bu}^{A,n} - E_{\bu}^{A,n-1 }$ in \eqref{estimate-eq1}
and $\phi_h = - E_{\psi}^{A,n}$ in \eqref{estimate-eq2} and adding the results, gives
\begin{align} \label{estimate-eq3} \nonumber
& a_1^h(E_{\bu}^{A,n}, E_{\bu}^{A,n}- E_{\bu}^{A,n-1}) + a_3(E_{\psi}^{A,n} - E_{\psi}^{A,n-1}, E_{\psi}^{A,n}) - b_2(E_{p}^{A,n} - E_{p}^{A,n-1}, E_{\psi}^{A,n}) \\
& = ( \bb(t_n) - \bb^n_h, E_{\bu}^{A,n}- E_{\bu}^{A,n-1 } )_{0, \Omega}  - b_1((\bu(t_n) - \bu(t_{n-1})) - (\Delta t) \partial_{t} \bu(t_n), E_{\psi}^{A,n})  \\
& \quad - b_2((I_p^h p^{\rmP}(t_n) - I_p^h p^{\rmP}(t_{n-1})) - (\Delta t) \partial_{t} p^{\rmP}(t_n), E_{\psi}^{A,n})  + a_3((I_{\psi}^h \psi(t_n) - I_{\psi}^h \psi(t_{n-1})) - (\Delta t) \partial_{t} \psi (t_n), E_{\psi}^{A,n}). \nonumber
\end{align}
Next, and as a consequence of using \eqref{weak-Iph},
\eqref{weak-Iph}, \eqref{weak-ph} and \eqref{weak-p} with $q_h^{\rmP} = E_p^{A,n}$, we are left with
\begin{align} \label{estimate-eq4}
& \tilde{a}_2^h(E_{p}^{A,n} - E_{p}^{A,n-1}, E_{p}^{A,n}) +  \Delta t  a_2^h(E_{p}^{A,n}, E_{p}^{A,n}) - b_2(E_{p}^{A,n}, E_{\psi}^{A,n} - E_{\psi}^{A,n-1}) \nonumber
\\ & = \Delta t ( \ell^{\rmP}(t_n)- \ell^{n,\rmP}_h, E_{p}^{A,n} )_{0, \OmP} + \tilde{a}_2^h( I^h_p p^{\rmP}(t_n) - I^h_p p^{\rmP}(t_{n-1}), E_{p}^{A,n}) \\
& \qquad -  \tilde{a}_2((\Delta t) \partial_{t} p^{\rmP}(t_n), E_{p}^{A,n}) + b_2(E_{p}^{A,n}, (\Delta t) \partial_{t} \psi - (I^h_{\psi} \psi(t_n)) - I^h_{\psi} \psi(t_{n-1})).\nonumber
\end{align}
Adding \eqref{estimate-eq3}-\eqref{estimate-eq4} and repeating the  arguments used in deriving stability, we can assert that 	
\begin{align*}
& a_3(E_{\psi}^{A,n} - E_{\psi}^{A,n-1}, E_{\psi}^{A,n}) -  b_2(E_{p}^{A,n} - E_{p}^{A,n-1}, E_{\psi}^{A,n})\\
& \qquad -  b_2(E_{p}^{A,n}, E_{\psi}^{A,n} - E_{\psi}^{A,n-1}) + \tilde{a}_2^h(E_{p}^{A,n} - E_{p}^{A,n-1}, E_{p}^{A,n}) \\
& = (\Delta t) \bigg( c_0(\delta_t E_{p}^{A,n} , E_{p}^{A,n})_{0,\OmP} + \frac{1}{\lambda} \sum_{K \in \cT_h^{\rmP}} \big(\alpha^2 (\delta_t (I- \Pi^{0,k}_K) E_p^{A,n}, (I- \Pi^{0,k}_K) E_p^{A,n})_{0,K} \\
& \qquad -(\delta_t (\alpha \Pi^{0,k}_K E_p^{A,n} - E_{\psi}^{A,n}), \alpha \Pi^{0,k}_K E_p^{A,n} - E_{\psi}^{A,n})_{0,K} \big) +  \frac{\Delta t}{\lambda} (\delta_t E_{\psi}^{A,n} , E_{\psi}^{A,n})_{0,\OmE} \bigg),
\end{align*}
The left-hand side can be bounded using the inequality
\[	(f_h^n - f_h^{n-1}, f_h^n) \ge \frac{1}{2} \bigl(\| f_h^n \|_0^2 - \| f_h^{n-1}\|_0^2 \bigr),
\]
and then summing over $n$ we get
\begin{align*} 
& {\mu_{\min}} \| \beps(E_{\bu}^{A,n})\|_0^2 + c_0 \| E_{p}^{A,n}\|_{0, \OmP}^2 + (1/ {\lambda^{\rmE}}) \|E_{\psi}^{A,n}\|_{0, \OmE}^2 + (\Delta t) \frac{\kappa_{\min}}{\eta} \sum_{j=1}^n \| \nabla E_{p}^{A,j} \|_{0, \OmP}^2 \nonumber \\
& \quad + (1/{\lambda^{\rmP}})  \sum_{K \in \cT_h^{\rmP}}\bigg( \alpha^2 \| (I- \Pi^{0,k}_K) E_p^{A,n} \|_{0,K}^2 + \| \alpha \Pi^{0,k}_K E_p^{A,n} - E_{\psi}^{A,n}\|_{0,K}^2 \bigg) \\
& \le {\mu_{\min}} \| \beps(E_{\bu}^{A,0})\|_0^2 + c_0 \| E_{p}^{A,0}\|_{0, \OmP}^2 + (1/ {\lambda^{\rmE}}) \|E_{\psi}^{A,0}\|_{0, \OmE}^2 \\
& \quad + (1/{\lambda^{\rmP}})  \sum_{K \in \cT_h^{\rmP}} \bigg( \alpha^2  \| (I- \Pi^{0,k}_K) E_p^{A,0} \|_{0,K}^2  + \| \alpha \Pi^{0,k}_K E_p^{A,0} - E_{\psi}^{A,0} \|_{0,K}^2 \bigg) \\
&  \quad + \underbrace{ \sum_{j=1}^n ( \bb(t_j) - \bb^j_h, E_{\bu}^{A,j}- E_{\bu}^{A,j-1} )_{0,\Omega}}_{=:L_1}
+ \underbrace{ \sum_{j=1}^n \Delta t ( \ell^{\rmP}(t_j)- \ell^{j,\rmP}_h, E_{p}^{A,j} )_{0,\OmP}}_{=:L_2} \nonumber \\
& \quad - \underbrace{ \sum_{j=1}^n b_1((\bu(t_j) - \bu(t_{j-1})) - (\Delta t) \partial_{t} \bu(t_j), E_{\psi}^{A,j})}_{=:L_3}\nonumber\\
&\quad - \underbrace{ \sum_{j=1}^n b_2((I_p^h p^{\rmP}(t_j) - I_p^h p^{\rmP}(t_{j-1})) - (\Delta t) \partial_{t} p^{\rmP}(t_j), E_{\psi}^{A,j})}_{=:L_4} \nonumber\\
& \quad + \underbrace{ \sum_{j=1}^n a_3((I_{\psi}^h \psi(t_j) - I_{\psi}^h \psi(t_{j-1})) - (\Delta t) \partial_{t} \psi (t_j), E_{\psi}^{A,j}) }_{:=L_5}  \nonumber\\
& \quad + \underbrace{ \sum_{j=1}^n (\tilde{a}_2^h( I^h_p p^{\rmP}(t_j) - I^h_p p^{\rmP}(t_{j-1}), E_{p}^{A,j}) -  \tilde{a}_2((\Delta t) \partial_{t} p^{\rmP}(t_j), E_{p}^{A,j}) )}_{:=L_6} \\
& \quad + \underbrace{ \sum_{j=1}^n b_2(E_{p}^{A,j}, (\Delta t) \partial_{t} \psi - (I^h_{\psi} \psi(t_j) - I^h_{\psi} \psi(t_{j-1}))}_{:=L_7}. \nonumber
\end{align*}
We bound the term $L_1$ with the help of formula
	\[
	\sum_{j=1}^n (f_h^j - f_h^{j-1}, g_h^j )= (f_h^n, g_h^n) - (f_h^0, g_h^0)-\sum_{j=1}^n (f_h^{j-1}, g_h^j- g_h^{j-1}),
	\] 
        the estimates of projection $\bPi^{0,k}_K$, applying Taylor expansion, and using the generalised Young's inequality. This gives
\begin{align*}
L_1 &
\lesssim \frac{{\mu_{\min}}}{2} \| \beps (E_{\bu}^{A,n})\|_0^2 +     \frac{h^k}{{\mu_{\min}}} |b(0)|_{k-1}~ {\mu_{\min}} \| \beps (E_{\bu}^{A,0})\|_0 + \frac{h^{2k}}{{\mu_{\min}}} \max_{1 \le j \le n} |\bb(t_j) |_{k-1}^2 \\
& \qquad
+ (\Delta t) ~ h^k \sum_{j=1}^n  \frac{1}{{\mu_{\min}}} \Big(| \partial_{t} \bb^j |_{k-1} +  \big(\Delta t \int_{t_{j-1}}^{t_j} | \partial_{tt} \bb (s)|_{k-1}^2 \, \mathrm{d}s \big)^{1/2} \Big) {\mu_{\min}} \| \beps(E_{\bu}^{A,j-1}) \|_0.
\end{align*}
Then the estimate satisfied by the projection $\Pi^{0,k}_K$ along with Poincar\'e and
Young's inequality, yield
\begin{align*}
L_2
& \lesssim \sum_{j=1}^n (\Delta t) \frac{\eta }{\kappa_{\min}} h^{2k} | \ell^{\rmP}(t_j)|_{k-1, \OmP}^2 +  (\Delta t) \frac{\kappa_{\min}}{6 \eta} \sum_{j=1}^n \| \nabla E_{p}^{A,j}\|_{0, \OmP}^2.
\end{align*}
The discrete inf-sup condition \eqref{discr-infsup} implies that
\begin{align} \label{bound:discrete-inf-sup}
\|E_{\psi}^{A,j} \|_0 \lesssim  h^k |\bb(t_j)|_{k-1} + {\mu_{\max}} \|\beps(E_{\bu}^{A,j}) \|_0 .
\end{align}
Applying an expansion in Taylor series, together with \eqref{bound:discrete-inf-sup},  the Cauchy--Schwarz, and Young inequalities, enable us to write
\begin{align*}
L_3
& \lesssim \sum_{j=1}^n  \Big( (\Delta t)^3  \int_{t_{j-1}}^{t_j} \| \partial_{tt} \bu (s)\|_0^2 \ds  \Big)^{1/2} (h^k |\bb(t_j)|_{k-1} + {\mu_{\max}} \|\beps(E_{\bu}^{A,j}) \|_0 ).
\end{align*}
Then, after using   \eqref{estimate-Ihp}, \eqref{bound:discrete-inf-sup}, and applying again Cauchy--Schwarz inequality, we get
\begin{align*}
L_4 
&  \lesssim \frac{\alpha}{{\lambda^{\rmP}}} \sum_{j=1}^n \bigg( h^{k+1}  \Big( (\Delta t) \int_{t_{j-1}}^{t_j}| \partial_{t}p^{\rmP}(s)|_{k+1, \OmP}^2\, \mathrm{d}s \Big)^{1/2} + \Big( (\Delta t )^{3} \int_{t_{j-1}}^{t_j} \| \partial_{tt} p^{\rmP}(s ) \|_{0, \OmP}^2\, \mathrm{d}s\Big)^{1/2} \bigg)  \\
& \qquad \qquad \qquad \times (  h^k |\bb(t_j)|_{k-1} + {\mu_{\max}} \|\beps(E_{\bu}^{A,j}) \|_{0, \OmP} ).
\end{align*}
On the other hand, the stability of $a_3(\cdot, \cdot)$ and the proof for the bound of $L_4$ gives
\begin{align*}
L_5 
& \lesssim \frac{1}{{\lambda_{\min}}} \sum_{j=1}^n \bigg( h^{k+1}  \Big( (\Delta t) \int_{t_{j-1}}^{t_j}(| \partial_{t}\bu^{\rmP}(s)|_{k+1,\OmP}^2 + | \partial_{t}\bu^{\rmE}(s)|_{k+1,\OmE}^2 + | \partial_{t}\psi^{\rmP}(s)|_{k, \OmP}^2 + | \partial_{t}\psi^{\rmE}(s) |_{k, \OmE}^2)\, \mathrm{d}s \Big)^{1/2} \\
& \qquad \qquad \qquad + \Big( (\Delta t)^3  \int_{t_{j-1}}^{t_j} \| \partial_{tt} \psi (s)\|_0^2 \ds  \Big)^{1/2}  \bigg) ( \rho h^k |\bb(t_j)|_{k-1} + {\mu_{\max}} \|\beps(E_{\bu}^{A,j}) \|_0 ).
\end{align*}
The polynomial approximation $p_\pi^{\rmP}$ for fluid pressure, consistency of the bilinear form $\tilde{a}_2^h(\cdot, \cdot )$, stability of the bilinear forms $\tilde{a}_2(\cdot, \cdot ), \tilde{a}_2^h(\cdot, \cdot )$, the Cauchy--Schwarz, Poincar\'e
and Young's inequalities gives		
\begin{align*}
L_6 
& \lesssim {\Big( c_0 +  \frac{\alpha^2}{{\lambda^{\rmP}}} \Big)^2} \Big( h^{2(k+1)} \|\partial_t p^{\rmP} \|_{L^2(0,t_n;H^{k+1}(\OmP))}^2 + (\Delta t)^2 \|  \partial_{tt} p^{\rmP} \|_{L^2(0,t_n;L^2(\OmP))}^2 \Big)
\\
& \qquad
+ \Delta t \frac{\kappa_{\min}}{6 \eta} \sum_{j=1}^n \| \nabla E_{p}^{A,j} \|_{0, \OmP}^2.
\end{align*}
The continuity of $b_2(\cdot, \cdot)$, the bound derived for the term $L_5$ and using the Young's inequality gives 	
\begin{align*}
L_7 
& \lesssim {\Big( \frac{\alpha}{{\lambda_{\min}}} \Big)^2}\bigg(  h^{2k} (\| \partial_{t} \psi \|_{L^2(0,t_n;k)}^2 + \| \partial_{t} \bu \|_{\bL^2(0,t_n;[H^{k+1}(\Omega)]^2)}^2)+ (\Delta t )^2 \| \partial_{tt} \psi \|_{L^2(0,t_n;L^2(\Omega))}^2 \bigg) \\
& \qquad
+(\Delta t) \frac{\kappa_{\min}}{6 \eta} \sum_{j=1}^n \| \nabla E_p^{A,j} \|_0^2.
%
\end{align*}
In turn, putting together the bounds obtained for all $L_i$'s, $i=1, \dots, 7$, using the Young's inequality and Lemma \ref{lem:Xn-bound} concludes that
\begin{align*}
&{\mu_{\min}} \| \beps(E_{\bu}^{A,n})\|_0^2 + c_0 \| E_{p}^{A,n}\|_{0, \OmP}^2  + (1/{\lambda^{\rmE}}) \| E_{\psi}^{A,n}\|_{0, \OmE}^2  + (\Delta t) \frac{\kappa_{\min}}{\eta} \sum_{j=1}^n \| \nabla E_{p}^{A,j} \|_{0, \OmP}^2 \\
& \lesssim   {\mu_{\min}} \| \beps(E_{\bu}^{A,0})\|_0^2 + (c_0 + \alpha^2/{\lambda^{\rmP}}) \| E_{p}^{A,0}\|_{0, \OmP}^2 + (1/{\lambda^{\rmE}}) \| E_{\psi}^{A,0}\|_{0, \OmE}^2  +  \Big(1 + \Delta t \Big) h^{2k}   \max_{0 \le j \le n} |\bb(t_j) |_{k-1}^2 \\
& \qquad + h^{2k} \Delta t \sum_{j=1}^n  (|\bb(t_j)|_{k-1}^2 + (\Delta t) | \partial_{t} \bb |_{k-1}^2 + | \ell^{\rmP}(t_j)|_{k-1, \OmP}^2)  + (\Delta t)^2 h^{2k} \| \partial_{tt} \bb \|_{L^2(0,t_n;[H^{k-1}(\Omega)]^2)} \\
& \qquad
+ (\Delta t )^2 \big( ( c_0 +  \alpha^2/{\lambda^{\rmP}} )^2 \| \partial_{tt} p^{\rmP} \|_{L^2(0,t_n;L^2(\OmP))}^2  + \| \partial_{tt} \bu \|_{\bL^2(0,t_n;[L^2(\Omega)]^2)}^2 \\
& \qquad + {\frac{\alpha^2}{{\lambda^2_{\min}}}} \| \partial_{tt} \psi \|_{L^2(0,t_n;L^2(\Omega))}^2 \big) +  h^{2k} \big( \Big(\frac{\alpha}{{\lambda_{\min}}}\Big)^2  \Big(\| \partial_{t} \psi \|_{L^2(0,t_n;k)}^2  +  \| \partial_{t} \bu\|_{\bL^2(0,t_n;[H^{k+1}(\Omega)]^2)}^2 \Big)
\\
& \qquad
+ {( c_0 +  \alpha^2/{\lambda^{\rmP}} )^2 } h^{2} \| \partial_{t}p^{\rmP} \|_{L^2(0,t_n;H^{k+1}(\OmP))}^2  \big) .
\end{align*}
And finally, the desired result \eqref{th4.3-est} holds after  choosing $\bu_h^0: =\bu_I(0)$,
$\psi_h^0: = \Pi^{0,k-1}\psi(0)$, $p_h^{0, \rmP}: = p_I^{\rmP}(0)$ and applying triangle's inequality together with \eqref{estimate-E_u^I}-\eqref{estimate-E_p^I} and \eqref{bound:discrete-inf-sup}.
\end{appendices}

\end{document}